\documentclass[a4paper]{amsart}
\usepackage{smfthm,amssymb}
\usepackage{hyperref}
\usepackage[all]{xy}
\usepackage{stmaryrd}
\usepackage{longtable}
\usepackage{bbm}
\usepackage[baseline]{euflag}

\SwapTheoremNumbers

\theoremstyle{plain}
\newtheorem{theorem}[subsubsection]{Theorem}
\newtheorem*{theorem*}{Theorem}
\newtheorem{lemma}[subsubsection]{Lemma}
\newtheorem{proposition}[subsubsection]{Proposition}
\newtheorem{corollary}[subsubsection]{Corollary}
\newtheorem{definition-proposition}[subsubsection]{Definition-Proposition}
\theoremstyle{definition}
\newtheorem{definition}[subsubsection]{Definition}
\newtheorem{example}[subsubsection]{Example}
\newtheorem{remark}[subsubsection]{Remark}

\newcommand{\Sp}{\mathrm{Sp}}
\newcommand{\SL}{\mathrm{SL}}
\newcommand{\MSp}{\mathbf{MSp}}
\newcommand{\MSL}{\mathbf{MSL}}
\newcommand{\MU}{\mathbf{MU}}
\newcommand{\etainv}{\eta^{-1}}
\newcommand{\MSpetainv}{{\MSp[\etainv]}}
\newcommand{\MSLetainv}{{\MSL[\etainv]}}
\newcommand{\E}{\mathcal{E}}
\newcommand{\W}{\mathcal{W}^{\eta}}
\newcommand{\HP}{\mathbb{HP}}
\newcommand{\HGr}{\mathrm{HGr}}
\newcommand{\HPinf}{\HP^{\infty}}
\newcommand{\HW}{\mathbf{H}\mathrm{W}}
\newcommand{\HMW}{\mathbf{H}\mathrm{MW}}
\newcommand{\KW}{\mathrm{KW}}
\newcommand{\KO}{\mathrm{KO}}
\newcommand{\kw}{\mathrm{kw}}
\newcommand{\SH}{\mathcal{SH}}
\newcommand{\ZZ}{\mathbb{Z}}
\newcommand{\QQ}{\mathbb{Q}}
\newcommand{\FF}{\mathbb{F}}
\newcommand{\coker}{\mathrm{coker}}
\newcommand{\cont}{\mathrm{cont}}
\newcommand{\univ}{\mathrm{univ}}
\newcommand{\homgr}{{\mathrm{Hom}^{gr}}}
\newcommand{\Rfr}{\mathcal{R}_{\mathrm{fr}}}

\renewcommand{\le}{\leqslant}
\renewcommand{\ge}{\geqslant}

\title{On $\eta$-periodic Formal Ternary Laws}
\author{Tao Huang}
\address{Institut Fourier - UMR 5582, Universit\'e Grenoble-Alpes, CNRS, CS 40700, 38058 Grenoble Cedex 9, France}
\email{tao.huang@univ-grenoble-alpes.fr}
\thanks{
	The author is fully supported by the COGENT project which has received funding from the European Union’s Horizon Europe Programme under the Marie Sklodowska-Curie actions HORIZON-MSCA-2023-DN-01 call (Grant agreement ID: 101169527), and from UK Research and Innovation.
}
\date{\today}

\begin{document}

\begin{abstract}
	We study the algebraic structure underlying $\Sp$-orientations in the $\eta$-periodic motivic stable homotopy category $\SH(k)[\etainv]$. Borel classes determine a geometric formal ternary law, but the $\HW$-Hurewicz map shows that its universal coefficients generate a proper $W(k)$-subalgebra $\Lambda\subsetneq (\MSpetainv)_*$, although $\Lambda[1/2]=(\MSpetainv)_*[1/2]$. Thus the failure of generation is purely $2$-primary. To capture part of the missing information, we introduce framed involutions. The spectrum $\MSpetainv$ carries a canonical framed involution, yielding a Quillen-type idempotent with telescope $\MSLetainv$ and a canonical splitting $(\MSpetainv)_* \cong \Rfr \otimes_{\ZZ} (\MSLetainv)_*$, where $\Rfr$ is the universal ring of framed involutions. We then axiomatize formal ternary laws, construct the universal Walter ring $\W$, and prove that $\W$ is isomorphic to the Lazard ring $L$ after inverting $2$. If $W(k)\cong \ZZ$, the universal geometric formal ternary law together with the canonical framed involution induces a classifying map $\phi:\W\to (\MSpetainv)_*$ that is injective and becomes an isomorphism after inverting $2$. Integrally, however, additional secondary power series are needed to recover the full cobordism coefficient ring.
\end{abstract}

\maketitle

\section{Introduction}\label{sec:introduction}

Quillen's theorem identifies the complex cobordism ring with the Lazard ring, that is, with the universal ring classifying one-dimensional commutative formal group laws \cite{quillen1969on}. A natural problem in symplectic cobordism is to determine the corresponding algebraic structure attached to $\Sp$-orientations.

In motivic homotopy theory, Panin and Walter developed the basic geometric input for this question by introducing quaternionic Grassmannians and Borel classes for symplectic bundles \cite{panin2021quaternionic}. They also constructed the algebraic symplectic and special linear cobordism ring spectra $\MSp$ and $\MSL$, and proved that maps $\MSp \to \E$ and $\MSL \to \E$ classify $\Sp$- and $\SL$-orientations on homotopy commutative ring spectra $\E$, respectively \cite{panin2023on}. Since the canonical ring map $\MSp\to \MSL$ induces an $\Sp$-orientation on every $\SL$-oriented spectrum, $\MSp$, together with the Borel classes of symplectic bundles, is the natural place to look for a symplectic analogue of the Lazard--Quillen picture.

The first Borel class of a rank-$2$ symplectic bundle plays the role of the first Chern class of a line bundle. However, the tensor product of two symplectic bundles is naturally orthogonal rather than symplectic, whereas the tensor product of three symplectic bundles is symplectic again. For this reason, the basic operation attached to an $\Sp$-orientation is not binary but ternary. The corresponding geometric object is the Borel polynomial of the triple tensor product of the three tautological rank-$2$ symplectic bundles on $(\HP^\infty)^{\times 3}$. This yields the geometric formal ternary law appearing in $\Sp$-oriented cohomology theories such as hermitian $K$-theory \cite{deglise2023the,fasel2025the}; see also Walter's notes \cite{walter2012formal}. Building on this geometric picture, Coulette, Déglise, Fasel, and Hornbostel introduced an axiomatic notion of formal ternary law and related it to Buchstaber's $2$-groups \cite{coulette2024formal}.

This point of view has substantial roots in classical topology. Shimakawa \cite{shimakawa1976remarks} formulated quaternionic oriented cohomology theories and, for any such theory $h^{*}$, defined a coefficient subring $\widetilde{\Lambda}_h$ generated by certain characteristic power series. In the universal case of symplectic cobordism, he showed that $\widetilde{\Lambda}_{\MSp}[1/2]$ is polynomial. Buchstaber then developed the framework for the Buchstaber subring generated by universal Pontryagin-type power series \cite{buchstaber1978characteristic}. Toda and Kozima subsequently introduced a symplectic formal system and its coefficient ring, the symplectic Lazard ring, as a symplectic analogue of the Lazard ring \cite{toda1982the}.

The $\eta$-periodic motivic category, obtained by inverting the stable motivic Hopf map $\eta\in \pi_{1,1}(\mathbbm{1})$, is particularly well suited to this question. Here $\eta$ is the stabilization of the morphism
\[
	\mathbb{A}^{2}\setminus 0 \to \mathbb{P}^{1},\qquad (x,y)\mapsto [x:y],
\]
using the identification $\mathbb{A}^{2}\setminus 0 \simeq S^{3,2}$ and $\mathbb{P}^{1}\simeq S^{2,1}$ in unstable motivic homotopy theory \cite{morel1999a1}. This already marks a basic difference between classical and motivic homotopy theory: classically the first Hopf map is nilpotent by Nishida's theorem \cite{nishida1973nilpotency}, whereas the motivic Hopf map $\eta$ is already non-nilpotent over fields \cite{morel2004pi0}.

Throughout the paper, we work in the $\eta$-periodic stable motivic homotopy category $\SH(k)[\etainv]$ over a Dedekind domain $k$.

The structural results of Bachmann and Hopkins, extended by Bachmann to Dedekind domains where $2$ is invertible, give a fiber sequence relating the $\eta$-periodic sphere to connective Witt $K$-theory and identify the coefficient rings of $\MSpetainv$ and $\MSLetainv$ as polynomial algebras over the Witt ring \cite{bachmann2021eta,bachmann2022eta}. Recent constructions of the motivic Hermitian $K$-theory spectrum $\KO$ make Hermitian $K$-theory available over more general bases \cite{calmes2024a,kumar2024construction}. With such constructions, we are able to make the following standing Bachmann--Hopkins assumption: the same $\eta$-periodic fiber sequence, the homotopy-sheaf and coefficient-ring identifications for $\MSpetainv$ and $\MSLetainv$, and the base-change compatibilities used in the reductions to the universal case. See also Remark~\ref{rem:odd-witt-classes}.

A first question of the paper is whether the coefficients of the universal geometric formal ternary law generate the entire $\eta$-periodic symplectic cobordism ring. The results can be summarized as follows.

\begin{theorem*}
	Let $\Lambda \subset (\MSpetainv)_*$ be the $W(k)$-subalgebra generated by the coefficients of the universal geometric formal ternary law. Then
	\[
		\Lambda \subsetneq (\MSpetainv)_*,
		\qquad
		\Lambda[1/2]=(\MSpetainv)_*[1/2].
	\]
	Equivalently, the coefficients of the universal geometric formal ternary law do not generate the $\eta$-periodic symplectic cobordism ring integrally, but they do after inverting $2$. The failure of generation is therefore purely $2$-primary.
\end{theorem*}

To study the missing coefficient data, we show that an $\Sp$-orientation carries not only a ternary law but also an involution induced by dualizing the universal symplectic bundle over $\HPinf$, together with a framing that trivializes it. We package this datum as a \emph{framed involution}. For $\MSpetainv$, the framed involution gives rise to the following splitting result.

\begin{theorem*}
	The coefficient ring $(\MSpetainv)_*$ carries a canonical framed involution. It induces an idempotent endomorphism of the homotopy commutative ring spectrum
	\[
		\xi:\MSpetainv\to \MSpetainv
	\]
	whose telescope is homotopy equivalent to $\MSLetainv$. Moreover, this induces a canonical isomorphism
	\[
		(\MSpetainv)_* \cong \Rfr \otimes_{\ZZ} (\MSLetainv)_*,
	\]
	where $\Rfr$ is the universal ring of framed involutions.\footnote{The universal ring $\Rfr$, and likewise the Walter ring introduced below, is obtained by quotienting the freely generated coefficient ring by the $2$-saturation of the naive coefficient ideal, coming from the involution equation for $\Rfr$ and from the FTL axioms for the Walter ring. This removes spurious $2$-torsion.}
\end{theorem*}

This theorem shows that framed involutions capture an essential part of the missing $2$-primary datum and therefore should be incorporated into the algebraic framework. Motivated by this, we introduce the notion of a \emph{formal ternary law} (FTL) in the $\eta$-periodic setting. We denote by $\W$ the universal $\Rfr$-algebra representing formal ternary laws, called the \emph{$\eta$-periodic Walter ring}. Passing to the trivial involution gives the normalized Walter ring $\widetilde{\W}:=\W\otimes_{\Rfr}\ZZ$. The datum of a formal ternary law splits into a framed involution and a normalized formal ternary law, and accordingly
\[
	\W\cong \Rfr\otimes_{\ZZ}\widetilde{\W}.
\]

The key structural point is that, after inverting $2$, formal ternary laws split and become equivalent to ordinary formal group laws. More precisely, over graded $\ZZ[1/2]$-algebras we construct natural transformations
\[
	\mathbb{T}:\mathcal{FGL}(-)\to\mathcal{FTL}(-),\qquad
	\mathbb{G}:\mathcal{FTL}(-)\to\mathcal{FGL}(-),
\]
and use them to compare the ternary and binary theories.
The construction of $\mathbb{G}$ also gives an explicit splitting of every FTL over a graded $\ZZ[1/2]$-algebra.
\begin{theorem*}[Lazard-type]
	There exists a universal $\eta$-periodic Walter ring $\W$ representing formal ternary laws. After inverting $2$, the natural transformations $\mathbb{T}$ and $\mathbb{G}$ identify formal ternary laws with ordinary formal group laws, and in particular $\mathbb{T}$ induces an isomorphism
	\[
		\W[1/2]\cong L[1/2].
	\]
\end{theorem*}

In other words, after inverting $2$, the ternary theory collapses to the classical theory of formal group laws. This is compatible with the Bachmann--Hoyois real-realization comparison \cite{bachmann2021norm}: if $\mathrm{R}\operatorname{Spec}(k)$ is a point (for example, $k=\ZZ$ or $\mathbb{R}$), then $\SH(k)[\etainv,1/2]\simeq \SH[1/2]$, and under real realization $\MSp[\etainv,1/2]$ is identified with $\MU[1/2]$. We then compare the universal Walter ring with the coefficient ring of $\MSpetainv$. The universal geometric formal ternary law determines classifying maps from $\W$.

\begin{theorem*}[Quillen-type]
	Suppose that $W(k)\cong \ZZ$ (for example, $k=\ZZ$ or $\mathbb{R}$). Then the universal geometric formal ternary law together with the canonical framed involution defines a classifying map
	\[
		\phi:\W\to (\MSpetainv)_*,
	\]
	and the normalized geometric formal ternary law over $\MSLetainv$ defines a classifying map
	\[
		\widetilde{\phi}:\widetilde{\W}\to (\MSLetainv)_*.
	\]
	Both maps are injective, and both become isomorphisms after inverting $2$.
\end{theorem*}

The map $\phi$ provides a first approximation to the desired $\Sp$-analogue of the classical pair $(L,\MU)$. Integrally, however, the classifying map $\phi$ is still not surjective, and the remaining missing $2$-primary information lies on the $\MSLetainv$ side and is encoded by additional secondary power series beyond those captured by a normalized formal ternary law. For reasons of length, the study of these additional secondary power series lies beyond the scope of the present paper.

The paper is organized as follows. Section~\ref{sec:formal-group-laws} recalls the classical Lazard--Quillen theory of formal group laws, and Section~\ref{sec:motivic-homotopy-theory} reviews the pieces of $\eta$-periodic motivic homotopy theory used later. In Section~\ref{sec:geometric-formal-ternary-law} we introduce the geometric formal ternary law, while Sections~\ref{sec:gftl-under-hurewicz} and~\ref{sec:comparing-gftl-and-msp} analyze its Hurewicz image and show that the $W(k)$-subalgebra generated by its coefficients is strictly smaller than $(\MSpetainv)_*$. Section~\ref{sec:framed-involution} defines framed involutions, constructs the canonical framed involution on $\MSpetainv$, and studies the associated Quillen-type idempotent. Section~\ref{sec:formal-ternary-law} axiomatizes formal ternary laws in the $\eta$-periodic setting, constructs the Walter ring and the natural transformations $\mathbb{T}$ and $\mathbb{G}$ relating FTLs and FGLs, and proves the Lazard-type and Quillen-type theorems. Appendix~\ref{app:gftl-hurewicz-computations} records the coefficient calculations used in the Hurewicz comparison.

\subsection*{Table of Notation}\label{subsec:notation-table}

\begingroup \small
\begin{longtable}{@{}lp{7.6cm}r@{}}
	\multicolumn{3}{@{}l}{\bfseries Rings and algebra}                                                                                                                                            \\[2pt]
	$QR,\;Q_n R$              & indecomposable quotient $R_+/R_+^2$, its degree-$n$ piece                                        &                                                                \\
	$[f]_g$                   & coefficient of monomial $g$ in polynomial $f$                                                    &                                                                \\
	$\cont_p(f)$              & $p$-local content of $f$ (gcd of coefficients localized at $p$)                                  &                                                                \\
	$\homgr(R,S)$             & graded ring homomorphisms $R\to S$                                                               &                                                                \\[6pt]
	\multicolumn{3}{@{}l}{\bfseries Formal group and ternary laws}                                                                                                                                \\[2pt]
	$L$                       & Lazard ring                                                                                      &                                                                \\
	$\iota(t), \Psi(t)$       & framed involution                                                                                & \S\ref{subsec:framed-involution-ring}                          \\
	$\Rfr$                    & universal ring with framed involution                                                            & Prop.~\ref{prop:universal-framed-involution-ring}              \\
	$\W,\;\widetilde{\W}$     & $\eta$-periodic Walter ring and its normalized quotient                                          & \S\ref{subsec:walter-ring}                                     \\
	$\mathbb{T},\;\mathbb{G}$ & natural transformations between FGL and FTL                                                      & \S\ref{subsec:from-fgl-to-ftl}, \S\ref{subsec:from-ftl-to-fgl} \\
	$\Lambda,\;\Lambda_l$     & $W(k)$-subalgebra of $(\MSpetainv)_*$ generated by coefficients of GFTL (resp.\ degree-$l$ part) & \S\ref{subsec:indecomposable-formulas}                         \\[6pt]
	\multicolumn{3}{@{}l}{\bfseries Motivic spectra and operations}                                                                                                                               \\[2pt]
	$\MSp, \MSL$              & motivic $\Sp$- and $\SL$-cobordism spectra                                                       & \cite{panin2023on}                                             \\
	$\HP^{n}, \HGr(d,n)$      & quaternionic projective space / Grassmannian                                                     & \cite{panin2021quaternionic}                                   \\
	$b_i(\mathfrak{U})$       & $i$-th Borel class of symplectic bundle $\mathfrak{U}$                                           & Def.~\ref{def:borel-classes}                                   \\
	$\HW,\;\HMW$              & unramified (Milnor-)Witt cohomology spectrum                                                     & \cite{deglise2025homological}                                  \\
	$\KW,\;\kw$               & Witt $K$-theory spectrum and its connective truncation                                           & \cite{bachmann2022eta}                                         \\
	$\beta$                   & Bott element in $\KW_*$ and its connective lift in $\kw_*$                                       & Prop.~\ref{prop:kw-to-hw}                                      \\
	$\KO$                     & Hermitian $K$-theory spectrum                                                                    & \cite{calmes2024a}
\end{longtable}
\endgroup

\subsection*{Convention}
Unless stated otherwise, all motivic spectra are taken over a fixed Dedekind domain $k$. Since we work in the $\eta$-periodic category $\SH(k)[\etainv]$, we use the single grading on motivic homotopy groups (cf.~\cite{ananyevskiy2015the}). For $\E\in \SH(k)[\etainv]$, we write
\[
	\E_*:=\underline{\pi}_*\E(k)=\Gamma(k,\underline{\pi}_*\E)
\]
for the global sections of its homotopy sheaf, and we suppress the base $k$ from the notation. For another Dedekind domain $k'$, we write $\underline{\pi}_*\E(k')$ or $\E_*(k')$.

\subsection*{Acknowledgments}
I am very grateful to my advisor Jean Fasel for introducing me to this topic and for the many valuable discussions. I also thank Alexey Ananyevskiy, Frédéric Déglise and Jens Hornbostel for helpful comments and suggestions.

\section{Formal Group Laws}\label{sec:formal-group-laws}
We recall the parts of the classical theory that will be used later. Standard references are \cite{kochman1996bordism,ravenel1986complex}.
\subsection{Lazard ring}
\begin{definition}
	Let $R$ be a graded (commutative) ring. A \emph{formal group law} over $R$ is a power series
	\[
		\mu(x,y)\in R\llbracket x,y\rrbracket, \quad |x|=|y|=-2
	\]
	satisfying
	\begin{enumerate}
		\item \textbf{(Identity)} $\mu(x,0)=x$ and $\mu(0,y)=y$,
		\item \textbf{(Symmetry)} $\mu(x,y)=\mu(y,x)$, and
		\item \textbf{(Associativity)} $\mu(\mu(x,y),z)=\mu(x,\mu(y,z))$.
	\end{enumerate}
\end{definition}
\begin{proposition}
	If $\mu$ is a formal group law over $R$, then there is a power series $\iota_\mu(t)\in R\llbracket t\rrbracket$ (called the \textit{formal inverse}) such that $\mu(t,\iota_\mu(t))=0$.
\end{proposition}

\begin{definition}
	Let $\mu$ be a formal group law over $R$. For a nonnegative integer $n$, the $n$-series $[n]_\mu(t)\in R\llbracket t\rrbracket$ is defined recursively by $[n+1]_\mu(t)=\mu(t,[n]_\mu(t))$ with $[0]_\mu(t)=0$. Moreover, define $[-n]_\mu(t):=\iota_{\mu}([n]_\mu(t))$.
\end{definition}
\begin{proposition}
	If $\mu$ is a formal group law over $R$, then for any integers $r_1, r_2 \in \ZZ$, we have
	\begin{enumerate}
		\item $[r_1 + r_2]_{\mu}(t) = \mu([r_1]_{\mu}(t), [r_2]_{\mu}(t))$, and
		\item $[r_1 r_2]_{\mu}(t) = [r_1]_{\mu}([r_2]_{\mu}(t))$.
	\end{enumerate}
\end{proposition}

\begin{proposition}
	There is a graded ring $L$ (called the Lazard ring), equipped with a formal group law $\mu^{L}$, that is initial among formal group laws: for every formal group law $\mu^{R}$ over a graded ring $R$, there exists a unique graded ring homomorphism $g:L\to R$ such that $g_{*}\mu^{L}=\mu^{R}$.
\end{proposition}
\begin{proof}[Sketch of proof]
	Start with the polynomial ring on symbols $a_{i,j}$ of degree $2i+2j-2$, and consider the universal power series
	\[
		x+y+\sum_{i,j\ge 1}a_{i,j}x^iy^j.
	\]
	Let $\mathcal{I}$ be the ideal generated by the relations expressing the unit, commutativity, and associativity axioms of formal group laws. Then
	\[
		L=\ZZ[a_{i,j}\mid i,j\ge 1]/\mathcal{I}
	\]
	carries a universal formal group law by construction.
\end{proof}

\begin{theorem}[Lazard-Fr{\"o}hlich {\cite{lazard1955sur,frohlich1968formal}}]\label{thm:lazard-structure}
	There are $y_n\in L_{2n}$ for $n\ge 1$ such that
	\[
		L \cong \ZZ[y_{1}, y_{2}, \ldots].
	\]
\end{theorem}
\begin{proof}[Proof of Theorem~\ref{thm:lazard-structure}]
	Consider the ring homomorphism
	\[
		\theta: L \to \ZZ[a_1, a_2, \ldots]
	\]
	classifying the formal group law
	\[
		\mu(x,y) = \exp(\log(x) + \log(y))
	\]
	over $\ZZ[a_1, a_2, \ldots]$, where
	\begin{equation}\label{eq:exp-series}
		\exp(t) = t + \sum_{i\ge 1} a_i t^{i+1} \in \ZZ[a_{1}, a_{2}, \ldots] \llbracket t \rrbracket
	\end{equation}
	and $\log(t)$ is the compositional inverse power series of $\exp(t)$, in the form
	\begin{equation}\label{eq:log-series}
		\log(t) = t + \sum_{i\ge 1} m_i t^{i+1}.
	\end{equation}

	For $i,j\ge 1$, the element $\theta(a_{i,j})$ is the coefficient of $x^iy^j$ in $\mu(x,y)$. Since $\log(x)+\log(y)$ has no mixed terms, the mixed coefficient of bi-degree $(i,j)$ comes, modulo decomposables in $\ZZ[a_1,a_2,\ldots]$, only from the summand $a_{i+j-1}(x+y)^{i+j}$ in the exponential. Therefore
	\[
		\theta(a_{i,j}) \equiv \binom{i+j}{i} a_{i+j-1}
		\pmod{(a_1,a_2,\ldots)^2}.
	\]
	By Lemma~\ref{lem:binomial-gcd} below, the gcd of $\{\binom{n+1}{i}\}_{1\le i\le n}$ equals $p$ if $n+1=p^k$ for a prime number $p$, and $1$ otherwise. Therefore
	\[
		T_{2n} := \mathrm{Im}(Q_{2n} \theta) = \begin{cases}
			p a_n \ZZ, & \text{if }n+1=p^{k}, \text{ for a prime } p, \\
			a_n \ZZ,   & \text{ otherwise}
		\end{cases}.
	\]

	By the comparison lemma \cite[Lemma A2.1.12]{ravenel1986complex}, $Q_{2n} \theta:Q_{2n} L \to T_{2n}$ is injective. By the definition of $T_{2n}$, $Q_{2n} \theta$ is surjective and thus it is an isomorphism. Choose $y_{n} \in L_{2n}$ such that $\theta(y_{n})$ projects to a generator of $T_{2n}$. Define a graded ring homomorphism $\gamma: \ZZ[Y_1,Y_2,\ldots] \to L$ by $\gamma(Y_{n}) = y_{n}$ for $n\ge 1$. Since $Q\gamma$ is an isomorphism and both rings are connected graded rings, a standard induction on degree shows that $\gamma$ is an isomorphism.
\end{proof}

\begin{lemma}\label{lem:binomial-gcd}
	For $n \ge 1$,
	\[
		\gcd\left\{\tbinom{n+1}{i} : 1 \le i \le n\right\}
		= \begin{cases}
			p & \text{if } n+1 = p^k \text{ for prime } p \text{ and } k \ge 1, \\
			1 & \text{otherwise.}
		\end{cases}
	\]
\end{lemma}
\begin{proof}
	By Kummer's theorem \cite{kummer1852uber}, $\nu_p\binom{n+1}{i}$ equals the number of carries when adding $i$ and $n+1-i$ in base $p$.

	Suppose $n+1 = p^k$.  Since $p^k$ has base-$p$ representation $1\underbrace{0\cdots 0}_{k}$, any sum $i + (p^k - i) = p^k$ with $1 \le i \le p^k - 1$ produces at least one carry, so $p \mid \binom{p^k}{i}$ for all such $i$.  Taking $i = p^{k-1}$, the addition $p^{k-1} + (p-1)p^{k-1}$ produces exactly one carry at position $k-1$, giving $\nu_p\binom{p^k}{p^{k-1}} = 1$.  For any prime $q \ne p$, one has $q \nmid \binom{p^k}{1} = p^k$, so $q \nmid \gcd$.  Hence $\gcd = p$.

	Suppose $n+1$ is not a prime power.  For any prime $p$, write $n+1 = p^a m$ with $p \nmid m$ and $m > 1$.  Then $i := p^a \le n$, and Lucas' theorem \cite{lucas1878the} gives $\binom{n+1}{p^a} \equiv m \pmod{p}$, which is nonzero.  Hence $p \nmid \gcd$, so $\gcd = 1$.
\end{proof}

\subsection{Geometric formal group law and Quillen's theorem}
\begin{definition}
	Let $\E$ be an oriented homotopy commutative ring spectrum, and let $t_{\E}\in \E^2(\mathbb{CP}^{\infty})$ be the chosen orientation. Let $m:\mathbb{CP}^{\infty}\times \mathbb{CP}^{\infty}\to \mathbb{CP}^{\infty}$ be the map classifying the tensor product of the universal line bundles. The associated formal group law over $\E_*$ is the unique power series
	\[
		\mu^{\E}(x,y)\in \E_*\llbracket x,y\rrbracket \cong \E^*(\mathbb{CP}^{\infty}\times \mathbb{CP}^{\infty})
	\]
	such that $m^*(t_{\E})=\mu^{\E}(x,y)$, where $x=\mathrm{pr}_1^*(t_{\E})$ and $y=\mathrm{pr}_2^*(t_{\E})$.

	In particular, for $\E=\MU$ with its canonical complex orientation, we write $\mu^{\MU}$ for the resulting formal group law over $\MU_*$.
\end{definition}
\begin{theorem}[Quillen \cite{quillen1969on}]
	The map $\phi:L \to \MU_{*}$ determined by the formal group law $\mu^{\MU}$ is an isomorphism. Moreover, the following diagram commutes.
	\[
		\xymatrix{
		L \ar[r]^-{\phi} \ar[d]_-{\theta} &  \MU_{*} \ar[d]_-{h}\\
		\ZZ[a_1, a_2, \ldots] \ar[r]^-{J} & H_*\MU
		}
	\]
	Here $h$ is the Hurewicz map and $J: \ZZ[a_1, a_2, \ldots] \to H_*\MU$ is the structural isomorphism.
\end{theorem}

\subsection{Universal base-change formula}
The following lemma is the topological prototype of the universal base-change formula that will later be used for $\MSpetainv$.

\begin{lemma}
	Let $\E$ be an oriented homotopy commutative ring spectrum, and set $A:=\E\wedge\MU$. Then
	\[
		A_{\ast} \cong \E_*[a_1,a_2,\ldots].
	\]
	Moreover, for either choice of induced orientation $t=t_{\E}$ or $t=t_{\MU}$, there is an identification
	\[
		A^{\ast}(\mathbb{CP}^{\infty}) \cong A_* \llbracket t \rrbracket \cong \E_*[a_1,a_2,\ldots]\llbracket t \rrbracket.
	\]
	In $A^{\ast}(\mathbb{CP}^{\infty})$, the two coordinates satisfy
	\[
		t_{\MU}=\exp(t_{\E})=t_{\E}+\sum_{i\ge 1}a_i t_{\E}^{i+1}.
	\]
	In particular, when $\E=H\ZZ$, the formal group law over $\MU$ satisfies
	\[
		h(\mu^{\MU}(x,y)) = \exp(\log(x) + \log(y)).
	\]
\end{lemma}
\begin{proof}
	The coefficient ring identification $A_*\cong \E_*[a_1,a_2,\ldots]$ is classical; see \cite[Part II, \S 4, Lemma 4.5(ii)]{adams1974stable}. The description of $A^*(\mathbb{CP}^{\infty})$ in terms of either coordinate $t_{\E}$ or $t_{\MU}$ is the usual complex-orientation coordinate description.

	The comparison of the two coordinates is proved in \cite[Lemma II.6.3]{adams1974stable}, using the Kronecker pairing between homology and cohomology of $\mathbb{CP}^{\infty}$. The specialization to $\E=H\ZZ$ is then \cite[Corollary II.6.6]{adams1974stable}.
\end{proof}
\begin{remark}
	Under the rational Hurewicz isomorphism
	\[
		\MU_* \otimes \QQ \xrightarrow[\cong]{h} H_*(\MU)\otimes \QQ,
	\]
	the power series \eqref{eq:exp-series} and \eqref{eq:log-series} admit unique lifts
	\[
		\exp^{\MU}(t), \log^{\MU}(t) \in \MU_* \otimes \QQ \llbracket t \rrbracket.
	\]
	Since $h(\mu^{\MU}(x,y))=\exp(\log(x)+\log(y))$, it follows that
	\[
		\mu^{\MU}(x,y)=\exp^{\MU}\bigl(\log^{\MU}(x)+\log^{\MU}(y)\bigr)
		\in \MU_* \llbracket x,y \rrbracket \subset \MU_* \otimes \QQ \llbracket x,y \rrbracket.
	\]

	More generally, let $\E$ be an oriented homotopy commutative ring spectrum with torsion-free coefficient ring, and let $g:\MU\to \E$ be the map classifying the chosen orientation. Define
	\[
		\exp^{\E}:=g_*(\exp^{\MU}), \qquad \log^{\E}:=g_*(\log^{\MU})
		\in \E_* \otimes \QQ \llbracket t \rrbracket.
	\]
	Applying $g_*$ to the identity above gives
	\[
		\mu^{\E}(x,y)=\exp^{\E}\bigl(\log^{\E}(x)+\log^{\E}(y)\bigr)
		\in \E_* \llbracket x,y \rrbracket \subset \E_* \otimes \QQ \llbracket x,y \rrbracket.
	\]
\end{remark}

\section{Motivic homotopy theory}\label{sec:motivic-homotopy-theory}
We now pass from classical stable homotopy theory to the motivic setting, more precisely to $\eta$-periodic motivic stable homotopy theory. We begin by recalling the $\Sp$-/$\SL$-orientation formalism developed by Panin and Walter \cite{panin2021quaternionic,panin2023on}, and then collect the structural results of Bachmann and Hopkins \cite{bachmann2021eta,bachmann2022eta} that will be used repeatedly throughout the sequel.

\subsection{Orientations and Borel classes}
\begin{definition-proposition}[{\cite{panin2023on}}]\label{def:sp-orientation}
Let $\E$ be a homotopy commutative ring spectrum in $\SH(k)$. An $\Sp$-orientation on $\E$ is any of the following equivalent data:
\begin{enumerate}
	\item A class $t_\E \in \E^{4,2}(\HPinf)$ whose pullback along $\HP^{1}\to\HPinf$ is the unit $1_{\E}\in \E^{4,2}(\HP^{1}) \cong \E_{0,0}$,
	\item a homomorphism of ring spectra $\MSp \to \E$.
\end{enumerate}
\end{definition-proposition}

\begin{definition}
	Let $\E$ be an $\Sp$-oriented homotopy commutative ring spectrum, let $X$ be a smooth $k$-scheme, and let $\mathfrak U$ be a rank-$2$ symplectic bundle over $X$. Choose a Jouanolou--Thomason affine replacement \cite[Proposition~4.4]{weibel1989homotopy} $p\colon\widetilde X\to X$ and let $f_{\mathfrak U}\colon\widetilde X\to\HPinf_{\widetilde X}$ be a section classifying $p^*\mathfrak U$ as in \cite[Equation~(11.2) and Theorem~11.2]{panin2023on}. Denote by $t_{\E,\widetilde X}\in\E^{4,2}(\HPinf_{\widetilde X})$ the pullback of $t_\E$ along the projection $\HPinf_{\widetilde X}\to\HPinf$, and define the \emph{first Borel class} of $\mathfrak U$ by
	\[
		b_1(\mathfrak U) := (p^*)^{-1}f_{\mathfrak U}^*(t_{\E,\widetilde X})
		\in\E^{4,2}(X),
	\]
	which is independent of the choices of the affine replacement and the classifying section.
\end{definition}

\begin{definition}[{\cite{panin2023on}}]\label{def:sl-orientation}
	Let $\E$ be a homotopy commutative ring spectrum in $\SH(k)$. An $\SL$-orientation on $\E$ is a homomorphism of ring spectra $\MSL\to \E$.
\end{definition}

\begin{remark}\label{rem:sl-induces-sp}
	The canonical map $\MSp\to \MSL$ induces, by composition, an $\Sp$-orientation on every $\SL$-oriented homotopy commutative ring spectrum.
\end{remark}

\begin{theorem}[Quaternionic projective bundle theorem, {\cite{panin2021quaternionic}}]\label{thm:quaternionic-projective-bundle}
	Let $\E$ be an $\Sp$-oriented homotopy commutative ring spectrum in $\SH(k)$, let $X$ be a smooth $k$-scheme, and let $\mathfrak{V}$ be a rank-$2r$ symplectic bundle over $X$. Let $\pi:\HP_X(\mathfrak{V})\to X$ be the associated quaternionic projective bundle. Denote by $\mathfrak U$ the tautological rank-$2$ symplectic subbundle on $\HP_X(\mathfrak{V})$, and write again
	\[
		\zeta:=b_1(\mathfrak U)\in \E^{4,2}(\HP_X(\mathfrak{V})).
	\]
	Then one has an isomorphism of $\E^{*,*}(X)$-modules
	\[
		\E^{*,*}(\HP_X(\mathfrak{V}))\cong \bigoplus_{i=0}^{r-1}\E^{*,*}(X)\cdot \zeta^i,
	\]
	where the coefficients are pulled back along $\pi$. Equivalently, there exist unique classes
	\[
		b_i(\mathfrak{V})\in \E^{4i,2i}(X),\qquad 1\le i\le r,
	\]
	such that
	\[
		\zeta^r-b_1(\mathfrak{V})\zeta^{r-1}+\cdots+(-1)^r b_r(\mathfrak{V})=0
	\]
	in $\E^{*,*}(\HP_X(\mathfrak{V}))$. In particular,
	\[
		\E^{*,*}(\HPinf) \cong \E^{*,*}\llbracket t_{\E} \rrbracket.
	\]
\end{theorem}

\begin{definition}\label{def:borel-classes}
	Let $\E$ be an $\Sp$-oriented homotopy commutative ring spectrum, and let $\mathfrak{V}$ be a rank-$2r$ symplectic bundle over a smooth $k$-scheme $X$. The uniquely determined classes
	\[
		b_i(\mathfrak{V})\in \E^{4i,2i}(X),\qquad 1\le i\le r,
	\]
	of Theorem~\ref{thm:quaternionic-projective-bundle} are called the \emph{Borel classes} of $\mathfrak{V}$. For rank-$2$ symplectic bundles, this recovers the first Borel class defined above.
\end{definition}

\subsection{\texorpdfstring{$\eta$}{η}-periodic motivic homotopy}

We henceforth write $\KO$ for the motivic Hermitian $K$-theory spectrum \cite{calmes2024a,kumar2024construction}, $\KW:=\KO[\etainv]$ for the Witt $K$-theory spectrum, and $\kw$ for its connective truncation \cite{bachmann2022eta}. Further, $\HW$ and $\HMW$ denote the spectra representing unramified Witt cohomology and unramified Milnor--Witt cohomology, respectively \cite{deglise2025homological}. Equivalently, $\HW$ is the coconnective truncation of $\kw$ \cite[Corollary~4.11]{bachmann2022eta}. Following \cite[Sections~3.2.5 and 6.3.2]{bachmann2021eta}, let $\beta\in\pi_4\KW$ be the Bott element. The connective-cover map $\kw\to\KW$ induces $\pi_4\kw\cong\pi_4\KW$, and $\beta$ also denotes its canonical connective lift. Finally, we write $\psi^3$ for the third Adams operation \cite{bachmann2021norm,fasel2025the} on $\KW_{(2)}$.

As noted in the introduction, the integral forms and base-change compatibilities used beyond the stated hypotheses of the cited Bachmann--Hopkins results are part of this standing Bachmann--Hopkins assumption.

\begin{proposition}\label{prop:kw-to-hw}
	The spectrum $\kw$ has coefficient ring
	\[
		\kw_*\cong W(k)[\beta],
	\]
	and there is a canonical equivalence
	\[
		\kw/\beta \simeq \HW.
	\]
	Moreover, the spectra $\kw$ and $\HW$ are compatible with base change among Dedekind domains.
\end{proposition}
\begin{proof}
	This follows from \cite[Lemma~4.4 and Theorem~4.5]{bachmann2022eta}.
\end{proof}

\begin{theorem}[Bachmann-Hopkins]\label{thm:kw-fiber-sequence}
	Let $\mathbbm{1}$ denote the motivic sphere spectrum. After $2$-localization there is a fiber sequence
	\[
		\mathbbm{1}_{(2)}[\etainv]\to \kw_{(2)} \xrightarrow{\varphi} \Sigma^{4}\kw_{(2)}.
	\]
	Here $\varphi$ is the connective lift of $\beta^{-1}(\psi^{3}-\mathrm{id})$.
\end{theorem}
\begin{proof}
	This follows from \cite[Theorem~4.12]{bachmann2022eta} and \cite[Theorem~7.8]{bachmann2021eta}.
\end{proof}

\begin{corollary}\label{cor:msp-kw-short-exact-sequence}
	Smashing the fiber sequence of Theorem~\ref{thm:kw-fiber-sequence} with $\MSp$ yields a short exact sequence
	\[
		0 \to \pi_* \MSp_{(2)}[\etainv] \to \kw_*\MSp_{(2)} \xrightarrow{\varphi} \kw_{*-4}\MSp_{(2)} \to 0.
	\]
\end{corollary}
\begin{proof}
	Applying $\pi_*(-\wedge \MSp)$ to Theorem~\ref{thm:kw-fiber-sequence} gives a long exact sequence. By the proof of \cite[Theorem 8.8]{bachmann2021eta}, $\varphi$ is surjective, so this is a short exact sequence.
\end{proof}

Before describing the operation $\varphi$, we record the universal polynomial structure on $\E_*\MSpetainv$.

\begin{proposition}\label{prop:e-msp-homology}
	Let $\E$ be an $\eta$-periodic $\Sp$-oriented homotopy commutative ring spectrum. Then there are classes $a_n\in \E_{2n}\MSpetainv$ for $n\ge 1$ such that
	\[
		\E_*\MSpetainv \cong \E_*[a_1,a_2,\ldots].
	\]
	Moreover, if the chosen $\Sp$-orientation on $\E$ is induced by an $\SL$-orientation, the canonical map $\MSp \to \MSL$ annihilates $a_{i}$ for odd $i$, and
	\[
		\E_*\MSLetainv \cong \E_*[a_2,a_4,\ldots].
	\]
	In particular,
	\[
		\kw_*\MSp \cong W(k)[\beta,a_1,a_2,\ldots],\qquad
		\HW_*\MSp \cong W(k)[a_1,a_2,\ldots].
	\]
\end{proposition}
\begin{proof}
	The first statement is \cite[Theorem~4.1]{bachmann2021eta}. The displayed specializations for $\E=\kw$ and $\E=\HW$ follow from Proposition~\ref{prop:kw-to-hw}.
\end{proof}

\begin{proposition}\label{prop:msp-msl-coefficients}
	There exist generators $y_i\in (\MSpetainv)_{2i}$ such that
	\[
		(\MSpetainv)_* \cong W(k)[y_1,y_2,\ldots].
	\]
	Moreover, the canonical map $\MSp \to \MSL$ annihilates $y_{i}$ for odd $i$, and
	\[
		(\MSLetainv)_* \cong W(k)[y_2,y_4,\ldots].
	\]
	In the category of $\MSpetainv$-modules, the canonical map induces an equivalence
	\[
		\MSpetainv/(y_1,y_3,\ldots)\simeq \MSLetainv.
	\]
\end{proposition}
\begin{proof}
	This follows from \cite[Proposition~5.6]{bachmann2022eta} and \cite[Theorem~8.7 and Corollary~8.9]{bachmann2021eta}.
\end{proof}

\begin{proposition}\label{prop:varphi-on-kw-msp}
	In $\kw_*\MSp [1/3]$, one has
	\[
		\varphi(\beta)=8,\qquad \varphi(a_1)=0,
	\]
	and for every $n\ge 2$,
	\[
		\varphi(a_n)\equiv 3(n-1)a_{n-2}\pmod{(\beta)}.
	\]
\end{proposition}
\begin{proof}
	On $\kw_*\MSp[1/3]$, the operation $\varphi$ is given by $\varphi=(\psi^3-\mathrm{id})/\beta$. By \cite[Theorem~3.1(2)]{bachmann2021eta}, one has $\psi^3(\beta)=9\beta$, hence $\varphi(\beta)=8$. Moreover, the computation in the proof of \cite[Proposition~4.3]{bachmann2021eta}, together with the integral formula for $\psi^3(b)$ from \cite[Remark~4.18]{bachmann2021eta}, gives
	\[
		\psi^3(a_n)=\sum_{k=0}^{\lfloor n/2\rfloor}
		\left(\binom{n+k}{k}-2\binom{n+k}{k-1}\right)(3\beta)^k a_{n-2k},
	\]
	where we use the convention $\binom{n}{-1}=0$. In particular, $\psi^3(a_1)=a_1$, so $\varphi(a_1)=0$. For $n\ge 2$, the terms with $k\ge 2$ are divisible by $\beta^2$, while the $k=1$ term is
	\[
		\left(\binom{n+1}{1}-2\binom{n+1}{0}\right)3\beta a_{n-2}=3(n-1)\beta a_{n-2}.
	\]
	Therefore
	\[
		\psi^3(a_n)\equiv a_n+3(n-1)\beta a_{n-2}\pmod{(\beta^2)},
	\]
	and dividing by $\beta$ yields
	\[
		\varphi(a_n)\equiv 3(n-1)a_{n-2}\pmod{(\beta)}.
	\]
\end{proof}

\begin{remark}\label{rem:odd-witt-classes}
	In contrast with the $2$-invertible case \cite[Theorem~1.5.22]{balmer2005witt}, odd Witt groups over Dedekind domains where $2$ is not invertible need not vanish. For example, over $k=\ZZ$ one has $W^{3}(\ZZ)\cong \ZZ/2$ (see \cite[Proposition~4.3.1]{ranicki1981exact}), giving an additional $2$-torsion class $\xi$ in homological degree $1$ and
	\[
		\pi_*\KW_{\ZZ}\cong \ZZ[\beta^{\pm 1},\xi]/(2\xi,\xi^2),
	\]
	cf.~\cite[Theorem 10.13]{kolderup2025hermitian}. This odd-degree class does not enter the even-degree GFTL coefficient calculations below. For statements about the full homotopy rings over such bases, one should also include this class.
\end{remark}

\section{Geometric formal ternary law}\label{sec:geometric-formal-ternary-law}

We now use the symplectic orientation formalism recalled above. In contrast with the $\mathrm{GL}$-oriented case, an $\Sp$-orientation does not lead to a binary formal group law on the first Borel class. The reason is that the tensor product of two symplectic bundles is naturally orthogonal, whereas the tensor product of three symplectic bundles is symplectic again. The appropriate replacement is therefore a ternary law.

\begin{definition}[\cite{coulette2024formal}]
	Let $R$ be a graded ring. We equip $R\llbracket x_1,\ldots,x_n\rrbracket[t]$ with the total grading extending that of $R$ and determined by
	\[
		|x_i|=-2,\qquad |t|=2.
	\]
	An \emph{$(m,n)$-series} over $R$ is a homogeneous element of degree zero in $R\llbracket x_1,\ldots,x_n\rrbracket[t]$, of the form
	\[
		G_t(x_1,\ldots,x_n)=1+\sum_{l=1}^m G_l(x_1,\ldots,x_n)t^l.
	\]
	This convention is the negative double of the grading used in \cite[\S~2.1.4]{coulette2024formal}, matching the homological grading on coefficient rings used below.

	We say that $G_t$ \emph{splits with roots $r_1,\ldots,r_m$} if the grading $|r_i|=-2$ and
	\[
		G_t(x_1,\ldots,x_n)=\prod_{i=1}^m (1+r_i(x_1,\ldots,x_n)t)
	\]
	in $R\llbracket x_1,\ldots,x_n\rrbracket[t]$.
\end{definition}

\begin{definition}[{\cite[Definitions~2.1.13 and 2.1.15]{coulette2024formal}}]
	\label{def:transport-of-series}
	Let $R$ be a graded ring, let $G_t$ be a composable $(m,n)$-series over $R$, and let $\Phi(x)\in R\llbracket x\rrbracket$ be homogeneous of degree $-2$, with $\Phi(x)=x+O(x^2)$. Write $(\Phi\circ G)_t$ for rootwise composition. If $G_t=\prod_i(1+r_i t)$ in a splitting algebra, then
	\[
		(\Phi\circ G)_t=\prod_i(1+\Phi(r_i)t).
	\]
	The \emph{transport} $\Phi_*G$ is the target of the strict isomorphism $\Phi:G\to\Phi_*G$, determined by
	\[
		(\Phi\circ G)_t(x_1,\ldots,x_n) = (\Phi_*G)_t(\Phi(x_1),\ldots,\Phi(x_n)).
	\]
	Equivalently, if $G_t=\prod_i(1+r_i(x_1,\ldots,x_n)t)$ in a splitting algebra, then
	\[
		(\Phi_*G)_t = \prod_i \left( 1+ \Phi\!\left( r_i(\Phi^{-1}(x_1),\ldots,\Phi^{-1}(x_n)) \right)t \right).
	\]
\end{definition}

\begin{definition}[Geometric formal ternary law]
	Let $S$ be a base scheme and let $\E \in \SH(S)$ be an $\Sp$-oriented homotopy commutative ring spectrum. Let $\mathfrak U_i$ be the tautological rank-$2$ symplectic bundle on the $i$-th factor of $(\HP^\infty_S)^{\times 3}$, and put
	\[
		x=b_1(\mathfrak U_1),\qquad y=b_1(\mathfrak U_2),\qquad z=b_1(\mathfrak U_3).
	\]
	The tensor product $\mathfrak U_1\otimes \mathfrak U_2\otimes \mathfrak U_3$ is a rank-$8$ symplectic bundle and hence is classified by a map $m:(\HP^\infty_S)^{\times 3}\to \HGr(4,\infty)$. Denote by $\mathfrak U$ the tautological rank-$8$ symplectic bundle on $\HGr(4,\infty)$, and define the $(4,3)$-series over $\E^{(2,1)*}(S)$
	\[
		F_t^\E(x,y,z)=1+\sum_{l=1}^4 F_l^\E(x,y,z)t^l
		\in \E^{(2,1)*}(S)\llbracket x,y,z\rrbracket[t]
	\]
	by
	\[
		F_l^\E(x,y,z):=m^*\left(b_l(\mathfrak U)\right),\qquad 1\le l\le 4.
	\]
	This $(4,3)$-series is called the \emph{geometric formal ternary law} (GFTL) of $\E$.
\end{definition}

\begin{remark}
	Since the construction is symmetric in the factors of $(\HP^\infty_S)^{\times 3}$, each $F_l^\E(x,y,z)$ is a symmetric power series in $x,y,z$. We write
	\[
		F_l^\E(x,y,z)=\sum_{a \ge b \ge c\ge 0} a^l_{a,b,c} \sigma(x^a y^b z^c),
	\]
	where $\sigma$ denotes the orbit sum under permutations of $(x,y,z)$.
\end{remark}

\begin{example}[{\cite[(3.2.4c)]{deglise2023the}}] \label{ex:gftl-over-hw}
	If $\E = \HW$, the spectrum representing unramified Witt cohomology, then
	\[
		\begin{aligned}
			F_{1}^{\HW}(x,y,z) & = 0,                              \\
			F_{2}^{\HW}(x,y,z) & = -2\sigma(x^2),                  \\
			F_{3}^{\HW}(x,y,z) & = -8xyz,                          \\
			F_{4}^{\HW}(x,y,z) & = \sigma(x^4) - 2\sigma(x^2 y^2).
		\end{aligned}
	\]
\end{example}

\begin{remark}
	In \cite{deglise2023the}, this formula is proved for fields $k$ of characteristic different from $2$ and $3$. Setting $\beta=0$ in the coefficient formula below recovers the same expression, which corresponds to the quotient $\kw\to\kw/\beta\simeq\HW$.
\end{remark}

\begin{example}[{\cite[Lemma 1.4]{ananyevskiy2017stable}, see also \cite[Theorem 6.6]{fasel2025the}}] \label{ex:gftl-over-kw}
	If $\E = \KW$ is the Witt $K$-theory spectrum, and $\beta \in \KW_{4}(k)$ denotes the Bott element, then
	\[
		\begin{aligned}
			F_{1}^{\KW}(x,y,z) & = \beta xyz,                                        \\
			F_{2}^{\KW}(x,y,z) & = \beta \sigma(x^2y^2) - 2\sigma(x^2),              \\
			F_{3}^{\KW}(x,y,z) & = \beta \sigma(x^3yz)-8xyz,                         \\
			F_{4}^{\KW}(x,y,z) & = \beta x^2y^2z^2 + \sigma(x^4) - 2\sigma(x^2 y^2),
		\end{aligned}
	\]
\end{example}

\section{GFTL under the Hurewicz map} \label{sec:gftl-under-hurewicz}
Write $F_t=F_t^{\MSpetainv}$ and let $h:\MSpetainv \to \HW \wedge \MSp$ be the $\HW$-Hurewicz map. The goal of this section is to determine the image of the coefficient rings of $F_t$ inside $\HW_*\MSp \cong W(k)[a_1,a_2,\ldots]$.
\subsection{Universal base change}
We first start without the $\eta$-periodicity. Let $\E \in \SH(k)$ be an $\Sp$-oriented homotopy commutative ring spectrum and set $A:=\E\wedge\MSp$. Then by \cite{panin2021quaternionic} and \cite{bachmann2021eta},
\[
	A^{*,*}(\HPinf) \cong A^{*,*} \llbracket t \rrbracket \cong \E^{*,*} [ a_1, a_2, \ldots ] \llbracket t \rrbracket,
\]
where the generator $t$ can be chosen to be the one induced by the $\Sp$-orientation of either $\E$ or $\MSp$. We denote these choices by $t_{\E}$ and $t_M$, respectively.

\begin{proposition}\label{prop:msp-universal-base-change}
	In $A^{*,*}(\HPinf)$, the two coordinates $t_{\E}$ and $t_M$ satisfy the universal base-change formula
	\[
		t_{M} = \exp(t_{\E}).
	\]
\end{proposition}
\begin{proof}
	Using the perfect Kronecker pairing \cite[Lemma~4.8]{bachmann2021eta}
	\[
		\langle-,-\rangle:A^{*,*}(\HPinf)\otimes_{A_{*,*}}A_{*,*}(\HPinf)\to A_{*,*},
	\]
	let $\beta_i \in A_{4i,2i}(\HPinf)$ be the class dual to $t_{\E}^i$, i.e.\ $\langle t_{\E}^j,\beta_i\rangle=\delta_{i,j}$.

	Write
	\[
		t_M=\sum_{n\ge 0} c_{n} t_{\E}^{n+1}
	\]
	with $c_n\in A_{4n,2n}$ for $n\ge 0$. Pairing with $\beta_{n+1}$ gives $\langle t_M,\beta_{n+1}\rangle=c_{n}$. By definition of the Kronecker pairing, $\langle t_M,\beta_{n+1}\rangle$ is represented by the composite
	\[
		\Sigma^{4n,2n} \mathbbm{1} \xrightarrow{\beta_{n+1}}
		A \wedge \Sigma^{-4,-2}\Sigma^{\infty}_{+}\HPinf
		\xrightarrow{\mathrm{id}\wedge t_M}
		A \wedge A
		\xrightarrow{\mu}
		A.
	\]
	By the construction of the polynomial generators in \cite[Theorem~4.1(1)]{bachmann2021eta}, this composite is precisely $a_{n}$, with $a_0=1$. Thus $c_n = \langle t_M,\beta_{n+1}\rangle=a_n$ for $n\ge 0$. Therefore
	\[
		t_M=t_{\E}+\sum_{n\ge 1} a_n t_{\E}^{n+1}=\exp(t_{\E}),
	\]
	as claimed.
\end{proof}

\begin{corollary}\label{coro:transport-base-change}
	For the universal rank-$2n$ symplectic bundle $\gamma_n$ over $\HGr(n,\infty)$,
	\[
		b_t^M(\gamma_n)=\bigl(\exp \circ\, b^\E\bigr)_t(\gamma_n)
	\]
	in the sense of Definition~\ref{def:transport-of-series}, where $b_t^\E$ and $b_t^M$ denote the Borel polynomials for the two $\Sp$-orientations on $A$ induced by $\E$ and by $\MSp$, respectively.
\end{corollary}
\begin{proof}
	The map
	\[
		\alpha^{*}:A^{*,*}(\HGr(n,\infty)) \to A^{*,*}((\HPinf)^{\times n}) \cong A^{*,*} \llbracket t^{\E}_1, t^{\E}_2, \ldots, t^{\E}_n \rrbracket
	\]
	is split injective and satisfies
	\[
		\alpha^*(b^{\E}_i(\gamma_n))=\sigma_i(t^{\E}_1,\ldots,t^{\E}_n),\qquad
		\alpha^*(b^{M}_i(\gamma_n))=\sigma_i(t^{M}_1,\ldots,t^{M}_n).
	\]
	By Proposition~\ref{prop:msp-universal-base-change}, $t^M_j=\exp(t^\E_j)$ for each $j$. Hence
	\[
		\alpha^*b_t^M(\gamma_n)
		=\prod_{j=1}^{n}\left(1+\exp(t_j^\E)t\right)
		=\alpha^*\left(\bigl(\exp \circ\, b^\E\bigr)_t(\gamma_n)\right).
	\]
	Since $\alpha^*$ is split injective, this proves the transport formula.
\end{proof}

Now we return to the $\eta$-periodic world. Let $\mathfrak{U}_1,\mathfrak{U}_2,\mathfrak{U}_3$ be the three tautological rank-$2$ symplectic bundles on $(\HPinf)^{\times 3}$. Their tensor product is classified by a map
\[
	(\HPinf)^{\times 3}\to\HGr(4,\infty).
\]
Let $x^H,y^H,z^H$ be the $\HW$-coordinates and put $x=\exp(x^H)$, $y=\exp(y^H)$, and $z=\exp(z^H)$. Pulling back the $n=4$ case of Corollary~\ref{coro:transport-base-change}, with $\E=\HW$, gives
\[
	h(F_t)(x,y,z)=\bigl(\exp \circ F^{\HW}\bigr)_t(x^H,y^H,z^H),
\]
which is precisely $h(F_t)=\exp_*F_t^{\HW}$. From direct calculation, $F_t^{\HW}$ splits with roots
\[
	\varepsilon_1 x^H+\varepsilon_2 y^H+\varepsilon_3 z^H,
\]
where $\varepsilon_j\in\{\pm 1\}$ and $\varepsilon_1\varepsilon_2\varepsilon_3=-1$. This gives an explicit factorization of the Hurewicz image of the universal GFTL.

\begin{proposition}\label{prop:hurewicz-gftl-roots}
	Let $x,y,z$ denote the generators induced by the universal $\Sp$-orientation on $\MSpetainv$. Then
	\[
		h(F_t)=\prod_{\varepsilon_1\varepsilon_2\varepsilon_3=-1}
		\left(1+\exp \left(\varepsilon_1 \log(x) + \varepsilon_2 \log(y) + \varepsilon_3 \log(z)\right)t\right).
	\]
	Equivalently, the roots of $h(F_t)$ are
	\[
		\exp \left(\varepsilon_1 \log(x) + \varepsilon_2 \log(y) + \varepsilon_3 \log(z)\right).
	\]
\end{proposition}
\begin{proof}
	Let $x^{H},y^{H},z^{H}$ denote the generators induced by the $\Sp$-orientation of $\HW$. By Proposition~\ref{prop:msp-universal-base-change}, we have
	\[
		x=\exp(x^{H}),\qquad y=\exp(y^{H}),\qquad z=\exp(z^{H}),
	\]
	so $x^{H}=\log(x)$, $y^{H}=\log(y)$, and $z^{H}=\log(z)$. Since $h(F_t)=\exp_*F_t^{\HW}$ and the roots of $F_t^{\HW}$ are $\varepsilon_1 x^{H}+\varepsilon_2 y^{H}+\varepsilon_3 z^{H}$ with $\varepsilon_1\varepsilon_2\varepsilon_3=-1$, Definition~\ref{def:transport-of-series} gives the stated product formula.
\end{proof}
\begin{remark}
	Before inverting $\eta$, the series $F_{t}^{\HMW}$ does not split in general.
\end{remark}

\subsection{Indecomposable Formulas}\label{subsec:indecomposable-formulas}
For $1\le l\le 4$, let $\Lambda_l \subset (\MSpetainv)_*$ be the $W(k)$-subalgebra generated by the coefficients of $F_l$, and let $\Lambda \subset (\MSpetainv)_*$ be the $W(k)$-subalgebra generated by all coefficients. We also set
\[
	S_n(x,y,z) := \sum_{\varepsilon_1 \varepsilon_2 \varepsilon_3 = -1} \left(\varepsilon_1 x + \varepsilon_2 y + \varepsilon_3 z\right)^{n},
\]
\[
	Q(x,y,z):=x^2+y^2+z^2,\qquad
	\Delta(x,y,z):=\sigma(x^4)-2\sigma(x^2y^2),
\]
\[
	P_m(x,y,z):=4(x^m+y^m+z^m)-S_m(x,y,z),
\]
\[
	C_n(x,y,z):=8\sigma(x^{n+1}yz)+S_{n+3}(x,y,z)-2Q(x,y,z)S_{n+1}(x,y,z),
\]
\[
	L_n(x,y,z):=\sigma\left(x^{n+2}(x^2-y^2-z^2)\right).
\]

\begin{proposition}\label{prop:hurewicz-indecomposable-formulas}
	One has
	\[
		h(F_1)=\sum_{n=1}^{\infty} a_n S_{n+1} \left(\log(x), \log(y), \log(z)\right).
	\]
	Moreover, modulo decomposables in the coefficient ring,
	\[
		\begin{aligned}
			h(F_1) & \equiv \sum_{n=1}^{\infty} a_n S_{n+1} \left(x, y, z\right),                                 \\
			h(F_2) & \equiv -2Q(x,y,z)+\sum_{n=1}^{\infty}a_n P_{n+2}(x,y,z),                                     \\
			h(F_3) & \equiv -8xyz+\sum_{n=1}^{\infty} a_n C_n(x,y,z),                                             \\
			h(F_4) & \equiv \Delta(x,y,z)+\sum_{n=1}^{\infty}a_n\left(\Delta(x,y,z)S_n(x,y,z)-4L_n(x,y,z)\right).
		\end{aligned}
	\]
\end{proposition}
\begin{proof}
	The first identity follows by expanding the product in Proposition~\ref{prop:hurewicz-gftl-roots}. The remaining formulas are obtained by taking the corresponding elementary symmetric polynomials in the four roots and then reducing modulo decomposables. The detailed expansions are recorded in Appendix~\ref{app:gftl-hurewicz-computations}.
\end{proof}

\subsection{Images of the Coefficient Rings}
To determine the image of the coefficient ring $\Lambda$, it suffices to compute the image $p$-primarily for each prime $p$. For an odd prime $p$, Remark~\ref{rem:lambda-image-after-inverting-two} below shows that $\Lambda_1 \subset \Lambda$ already controls the $p$-primary image. For $p=2$, we need to compute the $2$-primary image of each $\Lambda_l$.

\begin{theorem} \label{thm:lambda-one-hurewicz-image}
	For $n > 0$,
	\[
		\mathrm{Im} \left(Q_{2n} \Lambda_1 \xrightarrow{h} Q_{2n} \HW_*\MSp\right) = \begin{cases}
			p 2^{2+\nu_2(n)} a_n W(k), & n+1 = p^{k}, p \text{ odd prime}, \\
			2^{2+\nu_2(n)} a_n W(k),   & \text{otherwise.}
		\end{cases}
	\]
\end{theorem}
\begin{proof}
	This reduces to computing the content of the polynomials $S_{n+1}$ for $n\ge 1$; see Appendix~\ref{app:gftl-hurewicz-computations}.
\end{proof}

\begin{remark}\label{rem:lambda-image-after-inverting-two}
	For $k=\ZZ$, there are equivalences
	\[
		\MSpetainv[1/2] \simeq r_{\mathbb{R}}(\MSp)[1/2] \simeq \MU[1/2],
		\quad
		\HW[1/2] \simeq r_{\mathbb{R}}(\HW) \simeq \mathbf{H}\ZZ[1/2],
	\]
	see \cite[Lemma 4.7]{rondigs2019on}. Under these identifications, the following diagram commutes:
	\[
		\xymatrix{
		(\MSpetainv)_*[1/2] \ar[r]^-{\simeq} \ar[d]_-{h} & \MU_* [1/2] \ar[d]_-{h}\\
		\HW_*\MSp[1/2] \ar[r]^-{\simeq} & \mathbf{H}_*\MU[1/2]
		}
	\]
	Hence the two Hurewicz maps agree after inverting $2$. In particular, for $n > 0$,
	\[
		\begin{aligned}
			\mathrm{Im} \left(Q_{2n} (\MSpetainv)_*[1/2] \xrightarrow{h} Q_{2n} \HW_*\MSp[1/2] \right) \\
			\qquad = \begin{cases}
				         p a_n \ZZ, & n+1 = p^{k}, p \text{ odd prime}, \\
				         a_n \ZZ,   & \text{otherwise,}
			         \end{cases}
		\end{aligned}
	\]
	cf.~\cite{milnor1960on}. Comparing with Theorem~\ref{thm:lambda-one-hurewicz-image}, we conclude that
	\[
		\Lambda_1[1/2] = \Lambda[1/2] = (\MSpetainv)_*[1/2].
	\]
	Thus it remains to study only the higher-degree pieces localized at $2$.
\end{remark}

\begin{theorem}\label{thm:lambda-two-hurewicz-image}
	\[
		\mathrm{Im} \left(Q_{2n}\Lambda_2\xrightarrow{h}Q_{2n}\HW_*\MSp\right)_{(2)} = \begin{cases}
			8a_n W(k)_{(2)}, & \text{if }n+2=2^{r}, r\ge 2, \\
			4a_n W(k)_{(2)}, & \text{otherwise}.
		\end{cases}
	\]
\end{theorem}
\begin{proof}
	This follows from the content computation of the polynomials $P_m$; see Appendix~\ref{app:gftl-hurewicz-computations}.
\end{proof}

\begin{theorem}\label{thm:lambda-three-hurewicz-image}
	For $n\ge 1$,
	\[
		\mathrm{Im} \left(Q_{2n}\Lambda_3\xrightarrow{h}Q_{2n}\HW_*\MSp\right)_{(2)} = \begin{cases}
			4a_n W(k)_{(2)},  & \text{if $n$ is odd},            \\
			16a_n W(k)_{(2)}, & \text{if $n=2^r$ with $r\ge 2$}, \\
			8a_n W(k)_{(2)},  & \text{otherwise}.
		\end{cases}
	\]
\end{theorem}
\begin{proof}
	This is the corresponding $2$-local content computation for the polynomials $C_n$; see Appendix~\ref{app:gftl-hurewicz-computations}.
\end{proof}

\begin{theorem}\label{thm:lambda-four-hurewicz-image}
	\[
		\mathrm{Im} \left(Q_{2n}\Lambda_4\xrightarrow{h}Q_{2n}\HW_*\MSp\right)_{(2)} = \begin{cases}
			8a_n W(k)_{(2)}, & \text{if }n+2=2^{r}, r\ge 2, \\
			4a_n W(k)_{(2)}, & \text{otherwise}.
		\end{cases}
	\]
\end{theorem}
\begin{proof}
	This amounts to computing the $2$-local content of the polynomials $\Delta S_n-4L_n$; see Appendix~\ref{app:gftl-hurewicz-computations}.
\end{proof}

Combining the results above, we have
\begin{theorem} \label{thm:lambda-hurewicz-image-localized}
	\[
		\mathrm{Im} \left(Q_{2n} \Lambda \xrightarrow{h} Q_{2n} \HW_*\MSp\right)_{(2)} = \begin{cases}
			8 a_n W(k)_{(2)}, & \text{if }n + 2 = 2^{r}, r\ge 2, \\
			4 a_n W(k)_{(2)}, & \text{otherwise}.
		\end{cases}
	\]
\end{theorem}
\begin{proof}
	This is immediate from Theorems~\ref{thm:lambda-one-hurewicz-image},~\ref{thm:lambda-two-hurewicz-image},~\ref{thm:lambda-three-hurewicz-image}, and~\ref{thm:lambda-four-hurewicz-image}.
\end{proof}

\section{Comparing GFTL and \texorpdfstring{$\MSpetainv$}{MSp[η−1]}}\label{sec:comparing-gftl-and-msp}

In this section, we compute the image of the Hurewicz map
\[
	h: (\MSpetainv)_* \to \HW_*\MSp
\]
and then combine it with the results in the previous section to compare the rings $\Lambda$ and $(\MSpetainv)_*$.

\subsection{Image of the Hurewicz map}
Specifically, since the image of $h[1/2]$ is already known, we work with the image of $h$ localized at $2$.

\begin{theorem} \label{thm:msp-eta-inverted-hurewicz-image}
	For $n > 0$,
	\[
		\mathrm{Im} \left(Q_{2n} (\MSpetainv)_* \xrightarrow{h} Q_{2n} \HW_*\MSp\right)_{(2)} = \begin{cases}
			8 a_n W(k)_{(2)}, & \text{if }n=2             \\
			2 a_n W(k)_{(2)}, & \text{if }n=2^{k}, k\ge 2 \\
			a_n W(k)_{(2)},   & \text{otherwise.}
		\end{cases}
	\]
\end{theorem}

\begin{proof}
	By the base-change theorem of \cite[Theorem~4.5]{bachmann2022eta}, it suffices to determine the integers $m_n$ in the universal case over $k = \ZZ$. For a general base, the same coefficients are then transported by extension of scalars to $W(k)_{(2)}$.
	By Corollary~\ref{cor:msp-kw-short-exact-sequence}, there is a short exact sequence
	\[
		0 \to \pi_* \MSp_{(2)}[\etainv] \to \kw_*\MSp_{(2)} \xrightarrow{\varphi} \kw_{*-4}\MSp_{(2)} \to 0,
	\]
	and by Proposition~\ref{prop:kw-to-hw} the last map induced by $\kw \to \HW$ is the quotient by $\beta$. Thus the Hurewicz map $h$ is the composition
	\[
		\pi_* \MSp_{(2)}[\etainv] \to \kw_*\MSp_{(2)} \to \HW_*\MSp_{(2)}.
	\]

	The homotopy ring $\pi_* \MSp_{(2)}[\etainv]$ identifies with the subring
	\[
		\ZZ_{(2)}[y_1,y_2,y_3,\ldots] \subset \kw_*\MSp_{(2)} \cong \ZZ_{(2)}[\beta,a_1,a_2,a_3,\ldots],
	\]
	with $|y_i|=2i$. Set $R:=\ZZ_{(2)}[\beta,a_1,a_2,a_3,\ldots]$ and $R_{<n}:=\ZZ_{(2)}[\beta,a_1,\ldots,a_{n-1}]\subset R$. Each generator $y_n$ can be written as
	\[
		y_n=m_n a_n-h_n,\qquad m_n=2^{t_n},\qquad h_n\in R_{<n}.
	\]
	The image of $y_n$ in $Q_{2n} \HW_*\MSp_{(2)}$ is $m_n a_n$. Thus it suffices to show
	\[
		m_n= \begin{cases}
			8, & n=2,                      \\
			2, & n=2^k\text{ for }k \ge 2, \\
			1, & \text{otherwise.}
		\end{cases}
	\]

	Since $\varphi(y_n)=0$, we have $m_n\varphi(a_n)=\varphi(h_n)$, and both terms lie in $R_{<n}$. Hence $m_n$ is the index of $\varphi(a_n)$ in the graded cokernel
	\[
		\coker_n^R:=\coker \left(\varphi|_{R_{<n}}\right)_{2n-4}.
	\]

	For $n=1$ we have $\varphi(a_1)=0$ by Proposition~\ref{prop:varphi-on-kw-msp}, hence $m_1=1$. For $n=2$, the degree $4$ part $R_4$ is generated by $a_2$, $a_1^2$, and $\beta$. Since $\varphi(a_1^2)=0$ and $\varphi(\beta)=8$, we get $\coker_2^R\cong \ZZ/8\ZZ$. As $\varphi(a_2)=3$ by Proposition~\ref{prop:varphi-on-kw-msp}, the class $[\varphi(a_2)]$ has index $8$, so $m_2=8$.

	Now assume $n>2$. Since $\varphi$ is surjective, the graded map $R_{2n}\xrightarrow{\varphi}R_{2n-4}$ is surjective as well. Moreover,
	\[
		R_{2n}=(R_{<n})_{2n}\oplus \ZZ_{(2)}\{a_n\}, \qquad R_{2n-4}=(R_{<n})_{2n-4},
	\]
	so $\coker_n^R$ is cyclic generated by $[\varphi(a_n)]$. Equivalently, $\coker_n^R\cong \ZZ/m_n\ZZ$.

	First reduce modulo $2$. Let
	\[
		\bar R:=R/(2)=\FF_2[\beta,a_1,a_2,a_3,\ldots],
	\]
	and let $\bar\varphi$ be the induced map. Degree-wise reduction modulo $2$ is right exact, so
	\[
		\coker_n^{\bar R}:=\coker \left(\bar\varphi|_{\bar R_{<n}}\right)_{2n-4}\cong \coker_n^R\otimes_{\ZZ_{(2)}}\FF_2.
	\]
	Hence $\coker_n^{\bar R}=0$ if and only if $m_n=1$.

	Since $\varphi(\beta) = 8$ by Proposition~\ref{prop:varphi-on-kw-msp}, in $\bar R$ we have $\bar\varphi(\beta)=0$. By \cite[Lemma 8.3]{bachmann2021eta}, $\bar\varphi(\beta y)=\beta \bar\varphi(y)$, i.e.\ $\bar\varphi$ is $\beta$-linear. Since $\bar\varphi:\bar R_{2n-4}\to \bar R_{2n-8}$ is surjective, for every $z\in \bar R_{2n-8}$ choose $y\in \bar R_{2n-4}$ with $\bar\varphi(y)=z$. Then $\beta z=\bar\varphi(\beta y)$ lies in the image. Hence $\beta$ annihilates $\coker_n^{\bar R}$, so quotienting by $\beta$ does not change the cokernel. Thus we may work with $\bar S:=R/(2,\beta)=\FF_2[a_1,a_2,a_3,\ldots]$. The induced map, still denoted $\bar\varphi$, is a derivation and by Proposition~\ref{prop:varphi-on-kw-msp} satisfies
	\[
		\bar\varphi(a_n)= \begin{cases}
			a_{n-2}, & n \text{ even}, \\
			0,       & n \text{ odd},
		\end{cases}
	\]
	with the convention $a_0:=1$.

	Set
	\[
		\bar A:=\FF_2[a_2,a_4,\ldots]\subset \bar S,\qquad \bar I:=(a_1,a_3,\ldots)\subset \bar S.
	\]
	Then $\bar S_{<n}=\bar A_{<n}\oplus \bar I_{<n}$, and both summands are preserved by $\bar\varphi$. For each monomial $u$ in the odd generators of positive degree $2s$, the corresponding summand $\bar A u\subset \bar I$ is preserved by $\bar\varphi$, and in degree $2n-4$ the map is
	\[
		(\bar A_{<n})_{2(n-s)}u \to (\bar A_{<n})_{2(n-s)-4}u,\qquad xu\mapsto \bar\varphi(x)u.
	\]
	If $n-s$ is odd, the target vanishes. If $n-s$ is even, then $n-s<n$ and the extra generator $a_{n-s}\in \bar A_{<n}$ shows that this map is surjective. Therefore $\coker \left(\bar\varphi|_{\bar I_{<n}}\right)_{2n-4}=0$ and hence
	\[
		\coker_n^{\bar R}\cong \coker \left(\bar\varphi|_{\bar A_{<n}}\right)_{2n-4}.
	\]

	For either coefficient ring $\Lambda\in\{\FF_2,\ZZ/8\ZZ\}$, let $A_{\Lambda}:=\Lambda[a_2,a_4,\ldots]$. In the mod-$2$ case we keep the generators $a_{2m}$, and in the mod-$8$ case below we rescale them by units. Thus in both cases, writing $b_m$ for the resulting even generators, we may identify
	\[
		A_{\Lambda}\cong \Lambda[b_1,b_2,\ldots],\qquad d_{\Lambda}(b_m)=b_{m-1},\qquad b_0:=1,
	\]
	where $d_{\Lambda}$ denotes the induced map on the even subring. Up to doubling degrees this is the map $f_{\Lambda}$ in the short exact sequence
	\[
		0 \to H_*(\mathbf{BSU};\Lambda) \to H_*(\mathbf{BU};\Lambda) \xrightarrow{f_{\Lambda}} H_{*-2}(\mathbf{BU};\Lambda) \to 0,
	\]
	where $f_{\Lambda}$ is adjoint to cup product with $c_1$, cf.~\cite[Proposition 1.5]{kozima1987on}.

	Applying the same cokernel-coefficient argument to $d_{\FF_2}$ and using \cite[Theorem 3.3]{kochman1982polynomial}, we obtain
	\[
		\coker_n^{\bar R}= \begin{cases}
			\FF_2, & n=2^k,            \\
			0,     & \text{otherwise.}
		\end{cases}
	\]
	Thus $m_n=1$ whenever $n$ is not a power of $2$. It remains to treat $n=2^k$ with $k\ge 2$.

	For the exact value of $m_n$, reduce modulo $8$. Let
	\[
		S:=R/(8,\beta)=(\ZZ/8\ZZ)[a_1,a_2,a_3,\ldots],
	\]
	and let $\tilde\varphi$ be the induced map. Degree-wise reduction modulo $8$ is right exact, and the same $\beta$-quotient argument as above gives
	\[
		\coker \left(\tilde\varphi|_{S_{<n}}\right)_{2n-4}\cong \coker_n^R\otimes_{\ZZ_{(2)}}\ZZ/8\ZZ.
	\]
	By \cite[Lemma 8.3]{bachmann2021eta}, $\tilde\varphi$ is a derivation. Proposition~\ref{prop:varphi-on-kw-msp} implies that on generators
	\[
		\tilde\varphi(a_1)=0,\qquad \tilde\varphi(a_n)= \begin{cases}
			u_n a_{n-2}, & n \text{ even}, \\
			w_n a_{n-2}, & n \text{ odd},
		\end{cases}\qquad (n\ge 2),
	\]
	for some unit $u_n \in (\ZZ/8\ZZ)^{\times}$ and some $w_n \in 2(\ZZ/8\ZZ)$, with the convention $a_0:=1$.
	Set
	\[
		A:=(\ZZ/8\ZZ)[a_2,a_4,\ldots]\subset S,\qquad I:=(a_1,a_3,\ldots)\subset S.
	\]
	In particular, both $A$ and $I$ are preserved by $\tilde\varphi$, hence
	\[
		\coker \left(\tilde\varphi|_{S_{<n}}\right)_{2n-4}\cong \coker \left(\tilde\varphi|_{A_{<n}}\right)_{2n-4}\oplus \coker \left(\tilde\varphi|_{I_{<n}}\right)_{2n-4}.
	\]
	Since $n=2^k$ is even, $\tilde\varphi(a_n)\in A$, so the order of $[\tilde\varphi(a_n)]$ in the full cokernel is exactly its order in the even summand. Applying the same cokernel-coefficient argument to $d_{\ZZ/8\ZZ}$, and using Kochman's integral polynomial generators for $H_*(\mathbf{BSU})$ again, we obtain $\coker \left(\tilde\varphi|_{A_{<n}}\right)_{2n-4}\cong \FF_2$ for $n=2^k$ with $k\ge 2$. Since the mod-$2$ calculation already showed $m_n>1$, this forces $m_n=2$ for every such $n$. Therefore for $n \ge 3$
	\[
		m_{n}= \begin{cases}
			2, & n=2^k\text{ for }k \ge 2, \\
			1, & \text{otherwise.}
		\end{cases}
	\]
\end{proof}

\subsection{Insufficiency of GFTL}\label{subsec:insufficiency-gftl}
Combining Theorems~\ref{thm:lambda-hurewicz-image-localized} and~\ref{thm:msp-eta-inverted-hurewicz-image}, we obtain a strict inclusion
\[
	\Lambda \subsetneq (\MSpetainv)_*.
\]
Thus additional characteristic coefficients are needed to generate the full $\eta$-periodic symplectic cobordism ring, in the spirit of~\cite{toda1982the}.

However, the analogue of the additional datum used in \emph{loc.\ cit.} does not transfer directly to our setting. There, one uses the symplectification map $q:\mathbb{P}^{\infty} \to \HPinf$ classifying the canonical hyperbolic symplectic bundle $\mathcal{O}(-1)\oplus \mathcal{O}(1)$ on $\mathbb{P}^{\infty}$, and recovers the Borel class via pullback along $q^*$. In $\SH(k)[\etainv]$, this pullback vanishes for every $\SL$-oriented cohomology theory; see the projective bundle formula in~\cite{ananyevskiy2016on}. We therefore need a different enhancement of the GFTL.

\section{Framed involution}\label{sec:framed-involution}
\subsection{Ring with framed involution}\label{subsec:framed-involution-ring}
The relations we impose below come from geometric constructions on $\MSpetainv$. If one simply quotients by the naive defining relations to define the universal object, one introduces extra $2$-torsion that is not forced by the geometry. To avoid this artificial torsion, we systematically replace the naive defining ideals by their $2$-saturated closures. Since this will be used repeatedly in the definitions of framed involutions and formal ternary laws, we record the general notion once and for all.

\begin{definition}\label{def:2-saturated-relations}
	Let $P=\ZZ[c_i\mid i\in I]$ be a graded polynomial ring, and let $\mathcal{R}$ be a family of homogeneous polynomial relations in the variables $\{c_i\}_{i\in I}$. Denote by $\mathcal{I}(\mathcal{R})\subset P$ the homogeneous ideal generated by $\mathcal{R}$. Its \emph{$2$-saturated closure} is
	\[
		\mathcal{I}(\mathcal{R})^{\mathrm{sat}(2)}:=\mathcal{I}(\mathcal{R}):(2^\infty)=\{f\in P\mid \exists n\ge0,\ 2^n f\in \mathcal{I}(\mathcal{R})\}.
	\]
	Given homogeneous elements $\{r_i\}_{i\in I}$ of matching degrees in a graded ring $R$, we say that $\{r_i\}_{i\in I}$ \emph{satisfy the $2$-saturated relations} $\mathcal{R}$ if the induced graded ring map
	\[
		P\to R,\qquad c_i\mapsto r_i,
	\]
	factors through $P/\mathcal{I}(\mathcal{R})^{\mathrm{sat}(2)}$.
	In particular, since $\mathcal{I}(\mathcal{R})\subset \mathcal{I}(\mathcal{R})^{\mathrm{sat}(2)}$, every tuple satisfying the $2$-saturated relations $\mathcal{R}$ also satisfies the original relations $\mathcal{R}$.
\end{definition}

\begin{definition}
	Let $R$ be a graded ring. A \emph{framed involution} $(\iota, \Psi)$ over $R$ is a homogeneous power series of degree $-2$
	\[
		\Psi(t)\in R\llbracket t\rrbracket,\qquad |t|=-2,
	\]
	together with $\iota(t) := t-2\Psi(t)$ such that $\Psi(0)=0$, $\Psi'(0)=1$, and the coefficients of $\Psi(t)$ satisfy the $2$-saturated relations of the involution equation
	\[
		\iota(\iota(t))=t.
	\]
\end{definition}

\begin{remark}
	Every graded ring has a trivial framed involution $\iota(t) = -t, \Psi(t) = t$.
\end{remark}
\begin{remark}
	Let $(\iota, \Psi)$ be a framed involution over a graded ring $R$. Since $\Psi'(0)=1$, we have $\Psi(t)=t+O(t^2)$, so $\Psi$ is invertible under composition. Moreover, applying $\iota(t)=t-2\Psi(t)$ twice in the defining identity $\iota(t)=t-2\Psi(t)$ gives
	\[
		\iota(\iota(t))-t=-2\left(\Psi(\iota(t))+\Psi(t)\right)
	\]
	shows that every coefficient of $\Psi(\iota(t))+\Psi(t)$ belongs to the $2$-saturated closure of the naive involution ideal. Hence, for a framed involution,
	\[
		\Psi(\iota(t))=-\Psi(t),
	\]
	so
	\begin{equation}\label{eq:psi-trivializes-epsilon}
		\Psi\circ \iota\circ \Psi^{-1}(t)=-t,
	\end{equation}
	i.e.\ $\Psi$ is the base change that trivializes $\iota$.
\end{remark}
\begin{definition}
	Define the following functor from the category $\mathbf{Rings}^{\mathrm{gr}}$ of graded rings to the category $\mathbf{Sets}$ of sets:
	\[
		\mathcal{FI}(R):=\{(\iota, \Psi)\mid (\iota, \Psi)\text{ is a framed involution over }R\}.
	\]
\end{definition}

\begin{proposition}\label{prop:universal-framed-involution-ring}
	There exists a graded ring $\Rfr$ equipped with a framed involution $(\iota^{\univ}, \Psi^{\univ})$ such that for any graded ring $R$ with framed involution $(\iota^R, \Psi^R)$, there is a unique graded ring homomorphism $g:\Rfr \to R$ with
	\[
		g_*\Psi^{\univ}=\Psi^R.
	\]
	Equivalently, the functor $\mathcal{FI}$ is represented by $\Rfr$ equipped with universal framed involution $(\iota^{\univ}, \Psi^{\univ})$.
\end{proposition}
\begin{proof}
	Write
	\[
		\Psi^{\univ}(t)=t+\sum_{m\ge 1} b_m t^{m+1}, \qquad \iota^{\univ}(t) = t-2\Psi^{\univ}(t) = -t - 2\sum_{m\ge 1} b_m t^{m+1},
	\]
	with grading determined by $|b_m|=2m$. Define
	\[
		\Rfr:=\ZZ[b_1, b_2, b_3, \ldots]/\mathcal{J}^{\mathrm{sat}(2)},
	\]
	where $\mathcal{J}$ is the homogeneous ideal generated by the coefficients of
	\[
		\iota^{\univ}(\iota^{\univ}(t))-t,
	\]
	and $\mathcal{J}^{\mathrm{sat}(2)}$ is its $2$-saturated closure in the sense of Definition~\ref{def:2-saturated-relations}.

	Given a ring $R$ with framed involution $(\iota^R, \Psi^R)$, sending the generators $b_m$ to the corresponding coefficients of $\Psi^R$ defines a graded ring map $g:\Rfr \to R$, and by construction this $g$ is unique.
\end{proof}

\begin{proposition}
	The canonical map
	\[
		\ZZ[b_1,b_3,b_5,\ldots]\longrightarrow \Rfr
	\]
	is an isomorphism.
\end{proposition}
\begin{proof}

	First, we show that the even coefficients $b_{2k}$ are determined by the odd ones. Expanding the relation $\iota^{\univ}(\iota^{\univ}(t))=t$ modulo the ideal $I_{m-1} = (2b_1, \dots, 2b_{m-1})$ yields
	\[
		\iota^{\univ}(\iota^{\univ}(t)) \equiv t - ((-1)^{m+1}-1) 2 b_m t^{m+1} + O(t^{m+2}) \pmod{I_{m-1}}.
	\]
	For $m=2k$, the coefficient of $t^{2k+1}$ must vanish, forcing
	\[
		4b_{2k} = P_{2k}(2b_1, \dots, 2b_{2k-1}),
	\]
	for some polynomial $P_{2k}(c_1,\ldots,c_{2k-1}) \in \ZZ[c_1,\ldots,c_{2k-1}]$. By degree considerations for the coefficients $b_i$ appearing in the expansion of $\iota^{\univ}(\iota^{\univ}(t))$, every monomial appearing in the polynomial $P_{2k}(c_1,\ldots,c_{2k-1})$ contains at least two of the variables, that is,
	\[
		P_{2k}(c_1,\ldots,c_{2k-1}) \in (c_1,\ldots,c_{2k-1})^2.
	\]
	Therefore $P_{2k}(2b_1, \dots, 2b_{2k-1})$ is divisible by $4$, and by the $2$-saturation construction,
	\[
		b_{2k} = \frac{1}{4} P_{2k}(2b_1, \dots, 2b_{2k-1}) \in \Rfr,
	\]
	Thus $b_{2k}$ is uniquely determined by the preceding coefficients, so $\Rfr$ is generated entirely by the odd-indexed coefficients $\{b_{2k-1}\}_{k\ge 1}$. This proves the surjectivity of the map.

	To prove injectivity, we define an invertible power series
	\[
		f(t) := t + \sum_{k\ge 1} e_{2k-1} t^{2k}
	\]
	over $B:=\ZZ[e_1,e_3,e_5,\ldots]$. The power series $\iota^{B}(t):= f(-f^{-1}(t))$ is a compositional involution over $B$. Explicitly,
	\[
		\begin{aligned}
			\iota^{B}(t) & = -f^{-1}(t) + \sum_{k\ge 1} e_{2k-1} \left(- f^{-1}(t)\right)^{2k}            \\
			             & = -2f^{-1}(t) + f^{-1}(t) + \sum_{k\ge 1} e_{2k-1} \left(f^{-1}(t)\right)^{2k} \\
			             & = -2f^{-1}(t) + f\left(f^{-1}(t)\right)                                        \\
			             & = t - 2 f^{-1}(t).
		\end{aligned}
	\]
	Therefore,
	\[
		\Psi^{B}(t) := \frac{t - \iota^{B}(t) }{2} = f^{-1}(t)
	\]
	is integral, and $(\iota^{B}, \Psi^{B})$ is a framed involution over $B$. Note that
	\[
		[f^{-1}]_{t^{2k}} \equiv - e_{2k-1} \bmod (e_{1}, e_3, \ldots, e_{2k-3}).
	\]
	Then we have a map $\Rfr \to B$ by the universal property. The composition
	\[
		\ZZ[b_1,b_3,b_5,\ldots] \to \Rfr \to B
	\]
	is upper-triangular with diagonal $-1$ on its generators. Therefore, the map is injective.
\end{proof}

\subsection{Involution and idempotent over symplectic cobordism}
\begin{proposition}\label{prop:msp-framed-involution}
	The ring $(\MSpetainv)_*$ has a framed involution $(\iota^M, \Psi^M)$, where $\iota^{M}(t)$ is the power series associated with the first Borel class of the dual universal symplectic bundle.
\end{proposition}
\begin{proof}
	We need only to construct such a framing for $k=\ZZ$. The other cases are obtained by restricting the sheaf $\underline{\pi}_* \MSpetainv$. Indeed, set
	\[
		\Psi^{M}(t) := \frac{t - \iota^{M}(t)}{2}.
	\]

	To prove the $2$-divisibility of $t-\iota^M(t)$, it suffices to show that it vanishes after restricting to a field $k'$ such that $W(k') \cong \FF_2$ and the restriction map $W(\ZZ) \to W(k')$ identifies with reduction mod $2$ (for example, $k'=\mathbb{C}$). Then $\underline{\pi}_* \MSpetainv(k')$ is the kernel of the operation $\varphi: \kw_*\MSp(k') \to \kw_{*-4}\MSp(k')$, with inclusion $h:\underline{\pi}_* \MSpetainv(k') \hookrightarrow \kw_*\MSp(k')$ given by the Hurewicz map. By the universal base-change formula (cf.~Proposition~\ref{prop:msp-universal-base-change}),
	\[
		h(\iota^M)(t) = \exp( - \log(t)) =  - t + \sum_{n\ge 1} 2 a_{2n-1} \log^{2n}\left(t\right) = -t,
	\]
	as $\kw$ has involution $-t$. Since $W(k') \cong \FF_2$, this implies that $t-\iota^M(t)$ vanishes. Returning to the case $k=\ZZ$, we conclude that $\Psi^M(t) \in \underline{\pi}_* \MSpetainv(\ZZ)\llbracket t \rrbracket$ is well defined.
\end{proof}

\begin{theorem}[Quillen-type idempotent]\label{thm:quillen-idempotent}
	There exists a homotopy commutative ring endomorphism
	\[
		\xi:\MSpetainv\to\MSpetainv
	\]
	which is idempotent. Moreover, its telescope is homotopy equivalent to $\MSLetainv$.
\end{theorem}
\begin{proof}
	We shall prove this for $k=\ZZ$. The maps and homotopies over the other bases are obtained by base change. Using Definition~\ref{def:sp-orientation}, we define a new orientation $s \in \MSpetainv^*(\HPinf)$ by
	\[
		s = \Psi^{M}(t).
	\]
	This corresponds to a homotopy commutative ring endomorphism $\xi: \MSpetainv \to \MSpetainv$. The associated involution $\iota'(s)$ of this orientation is $-s$ by \eqref{eq:psi-trivializes-epsilon}. Applying $\xi_*$ to $2s=t-\iota^M(t)$ therefore gives $2\xi_*(s)=s-\iota'(s)=2s$. By Proposition~\ref{prop:msp-msl-coefficients} and the quaternionic projective bundle theorem, the even-degree coefficient groups occurring here are torsion-free over $\ZZ$, so $\xi_*(s)=s$. Hence $\xi^2$ and $\xi$ classify the same $\Sp$-orientation, and the universality of $\MSp$ gives $\xi^2\simeq\xi$.

	By \cite[Paragraph~5.19, Definition~6.1 and Corollary~6.2]{deglise2021on}, the unit of $\SH(k)[\etainv]_{\QQ}$ identifies with $\HW_{\QQ}$. Smashing with $\MSp$ and passing to coefficients gives the rational Hurewicz isomorphism
	\[
		h: (\MSpetainv)_*\otimes \QQ \xrightarrow{\cong} \HW_*\MSp\otimes \QQ,
	\]
	and the target has generators $\{a_i\}_{i\ge 1}$. Let $\log^{M}(t)$ and $\exp^{M}(t) \in (\MSpetainv)_*\otimes \QQ \llbracket t \rrbracket$ be the preimages of $\log(t)$ and $\exp(t)$, respectively. Then we have
	\[
		\iota^M(t) = \exp^M(-\log^M(t)).
	\]
	The exponential power series $\exp'(s)$ associated to the new orientation is
	\[
		\exp'(s) = s + \sum_{n\ge 1} a_{2n} s^{2n+1} \in (\MSpetainv)_*\otimes \QQ \llbracket s \rrbracket.
	\]
	This implies that $\xi_*\otimes \QQ: (\MSpetainv)_*\otimes \QQ \to (\MSpetainv)_*\otimes \QQ$ sends
	\[
		a_i \mapsto \begin{cases}
			a_i, & i \text{ even,} \\
			0,   & i \text{ odd.}
		\end{cases}
	\]

	Choose polynomial generators
	\[
		(\MSpetainv)_* \cong \ZZ[y_1,y_2,\ldots]
	\]
	as in Proposition~\ref{prop:msp-msl-coefficients}, and let
	\[
		q:\MSpetainv\to\MSLetainv
	\]
	be the canonical map. By the description $\Rfr\cong\ZZ[b_1,b_3,\ldots]$ above, $\Rfr$ is generated in degrees $4m-2$, whereas $(\MSLetainv)_{4m-2}=0$. Thus $q$ sends $\Psi^M$ to the trivial frame, so $q_*(s)=q_*(t)$. The maps $q\xi$ and $q$ therefore classify the same $\Sp$-orientation, and hence $q\xi\simeq q$. The following diagram commutes
	\[
		\xymatrix{\MSpetainv \ar[r]^-{q} \ar[d]_-{h} & \MSLetainv \ar[d]_-{h} \\
		\HW \wedge \MSp \ar[r]^-{q} & \HW \wedge \MSL }
	\]
	After tensoring with $\QQ$ and using the Hurewicz coordinates, Proposition~\ref{prop:e-msp-homology} identifies $q_*: \HW_* \MSp\otimes \QQ \to \HW_* \MSL \otimes \QQ$ with the homomorphism annihilating $a_{i}$ for $i$ odd. Since $\ker(q_*\otimes\QQ)$ is generated by the odd $y_i$, the rational Hurewicz isomorphism gives
	\[
		(y_1,y_3,\ldots)\otimes\QQ=(a_1,a_3,\ldots) = \ker(\xi_*\otimes\QQ)
		\subset (\MSpetainv)_*\otimes\QQ .
	\]
	Since $(\MSpetainv)_{4m-2}$ is torsion-free over $\ZZ$, it follows that $\xi_*(y_{2m-1})=0$ for all $m\ge 1$.
	Since $\xi$ is idempotent, there are splitting maps
	\[
		\MSpetainv\xrightarrow{p}\operatorname{Tel}(\xi)\xrightarrow{i}\MSpetainv,
		\qquad pi\simeq\mathrm{id}_{\operatorname{Tel}(\xi)},\quad ip\simeq\xi.
	\]
	Set $\omega:=qi:\operatorname{Tel}(\xi)\to\MSLetainv$. Since homotopy sheaves commute with filtered homotopy colimits, $\underline{\pi}_*\operatorname{Tel}(\xi)\cong\operatorname{im}(\xi_*)$ and $\omega_*$ is the restriction of $q_*$ to this image. By
	Proposition~\ref{prop:msp-msl-coefficients}, $\omega$ is an equivalence. Defining
	\[
		\bar{\xi}:=i \omega^{-1}:\MSLetainv\longrightarrow\MSpetainv
	\]
	gives $q\bar{\xi}\simeq\mathrm{id}_{\MSLetainv}, \omega p\simeq q\xi\simeq q$, and $\bar{\xi}q\simeq ip\simeq\xi$. Thus $\operatorname{Tel}(\xi)\simeq\MSLetainv$.
\end{proof}
\begin{remark}
	This is inspired by the splitting of $\MU[1/2]$ described in \cite{baker2014msp}, and readers can also find there some discussion of the relation between involutions and such splittings. However, such a splitting in the $\mathrm{GL}$-oriented case does not occur integrally.
\end{remark}

\begin{lemma}\label{lem:canonical-msp-splitting}
	Let
	\[
		\kappa:\Rfr\longrightarrow (\MSpetainv)_*
	\]
	be the map classifying the framed involution $\Psi^M$. Then multiplication induces an isomorphism
	\[
		\chi:\Rfr\otimes_{\ZZ}(\MSLetainv)_*\xrightarrow{\cong}(\MSpetainv)_*,
		\qquad
		r\otimes x\longmapsto \kappa(r)\cdot \bar{\xi}_*(x).
	\]
\end{lemma}
\begin{proof}
	As in Proposition~\ref{prop:msp-framed-involution}, it suffices to prove the claim for $k=\ZZ$. The other cases are obtained by restricting the sheaves $\underline{\pi}_*\MSpetainv$ and $\underline{\pi}_*\MSLetainv$. Thus $W(k)=\ZZ$.

	By the description $\Rfr\cong \ZZ[b_1,b_3,b_5,\ldots]$ above and Proposition~\ref{prop:msp-msl-coefficients}, both sides are polynomial rings over $\ZZ$ with one indecomposable generator in each positive degree. Hence it suffices to prove that $\chi$ induces an isomorphism on indecomposables.

	In degree $4m$, the map $Q_{4m}\chi$ is just $Q_{4m}\bar{\xi}_*$, because $\Rfr$ has no positive even indecomposables. Since $q_*\circ \bar{\xi}_*=\mathrm{id}_{(\MSLetainv)_*}$, the maps $Q_{4m}q_*$ and $Q_{4m}\bar{\xi}_*$ are inverse isomorphisms.

	In degree $4m-2$, the map $Q_{4m-2}\chi$ is just $Q_{4m-2}\kappa$, because $(\MSLetainv)_*$ has no odd indecomposables. Write
	\[
		\Psi^M(t)=t+\sum_{n\ge 1} c_n t^{n+1}.
	\]
	By the proof of Proposition~\ref{prop:msp-framed-involution},
	\[
		h(\Psi^M)(t) = \frac{t-h(\iota^M)(t)}{2} = t-\sum_{n\ge 1}a_{2n-1}\log^{2n}(t).
	\]
	Since $\log(t)=t+O(t^2)$, it follows that
	\[
		h(c_{2m-1})\equiv -a_{2m-1}\pmod{(a_1,\ldots,a_{2m-2})}.
	\]
	Therefore $Q_{4m-2}(h\circ \kappa)$ sends the generator of $Q_{4m-2}\Rfr\cong \ZZ$ to $-a_{2m-1}$. $y_{2m-1}$ is a generator of the free rank-one $\ZZ$-module $Q_{4m-2}(\MSpetainv)_*$. Writing
	\[
		Q_{4m-2}\kappa(b_{2m-1})=\lambda y_{2m-1},\qquad
		Q_{4m-2}h(y_{2m-1})=\mu a_{2m-1}
	\]
	for some $\lambda,\mu\in \ZZ$, the above congruence gives $\lambda\mu=-1$. Hence $\lambda=\pm 1$, so $Q_{4m-2}\kappa$ is an isomorphism.

	Thus $Q\chi$ is an isomorphism, and therefore $\chi$ is an isomorphism.
\end{proof}

\section{Formal Ternary Laws}\label{sec:formal-ternary-law}
\subsection{Axioms of Formal Ternary Laws}\label{subsec:axiom-formal-ternary-law}
In this subsection, we give an axiomatic definition of formal ternary laws in the $\eta$-periodic setting. Compared with \cite{coulette2024formal}, the scalar $\epsilon$ is identified with $1$ here.
\begin{definition}\label{def:formal-ternary-law}
	A \emph{formal ternary law} (FTL) over a graded ring $R$ equipped with a framed involution $(\iota, \Psi)$ is a $(4,3)$-series with coefficients in $R$:
	\[
		F_t(x,y,z)= 1 + \sum_{l=1}^{4} F_l(x,y,z)t^l,
	\]
	whose coefficients satisfy the $2$-saturated relations of the following axioms in the sense of Definition~\ref{def:2-saturated-relations}:
	\begin{enumerate}
		\item \textbf{(Neutral element)} The $(4,1)$-series satisfies
		      \[
			      F_t(x,0,0)=(1+xt)^2 (1+\iota(x)t)^2.
		      \]
		\item \textbf{(Symmetry)} $F_t$ is $S_3$-invariant:
		      \[
			      F_t(x_1,x_2,x_3)=F_t(x_{\sigma(1)},x_{\sigma(2)},x_{\sigma(3)})\qquad (\sigma\in S_3).
		      \]
		\item \textbf{(Associativity)} For formal variables $(x,y,z,u,v)$ one has an equality of $(16,5)$-series
		      \[
			      F_t\left(F_t(x,y,z),u,v\right)=F_t\left(x,F_t(y,z,u),v\right),
		      \]
		      where the substitution is as in \cite[Definition~2.1.13]{coulette2024formal}.
		\item \textbf{(Framed $\iota$-linearity)} Writing $\widehat{F}_t:=\Psi_*F_t$ for the transport of $F_t$ along $\Psi$, one has
		      \[
			      \widehat{F}_t(-x,y,z)=\widehat{F}_{-t}(x,y,z).
		      \]
		\item \textbf{(Weak neutral element)} The $(4,2)$-series $F_t(x,x,y)$ has a multiple root $y$, that is,
		      \[
			      F_t(x,x,y) = (1+yt)^2(1+ S(x,y)t + P(x,y)t^2),
		      \]
		      for some $S(x,y), P(x,y) \in R\llbracket x,y \rrbracket$.
	\end{enumerate}
\end{definition}
\begin{definition}
	An FTL is \emph{normalized} if its framed involution is trivial, i.e.\ $\iota(t)=-t$ and $\Psi(t)=t$. Framed $\iota$-linearity is then $F_t(-x,y,z)=F_{-t}(x,y,z)$. For every FTL, the coordinate transport $\widehat F_t:=\Psi_*F_t$ is a normalized FTL.
\end{definition}

\begin{remark}\label{rem:epsilon-linearity-consequence}
	By framed $\iota$-linearity and symmetry, $\widehat{F}_t(-x,-y,z)=\widehat{F}_t(x,y,z)$. Transporting this identity back along $\Psi^{-1}$ and using $\iota=\Psi^{-1}(-\Psi)$ gives
	\[
		F_t(\iota(x),\iota(y),z)=F_t(x,y,z).
	\]
\end{remark}
\begin{definition}
	Define the following functors from the category $\mathbf{Rings}^{\mathrm{gr}}$ of graded rings to the category $\mathbf{Sets}$ of sets:
	\[
		\mathcal{FGL}(R):=\{\mu \mid \mu \text{ is an FGL over }R\}.
	\]
	\[
		\mathcal{FTL}(R):=
		\left\{
		(\iota, \Psi, F_t)\,\middle|\,
		\begin{array}{l}
			F_t \text{ is an FTL over }R, \\
			\text{with framed involution }(\iota, \Psi)
		\end{array}
		\right\}.
	\]
	\[
		\mathcal{FTL}_{\mathrm{norm}}(R):=\{(\iota, \Psi, F_t) \in \mathcal{FTL}(R) \mid \Psi(t) = t \}.
	\]
\end{definition}

\begin{proposition}
	The GFTL $F_t$ over any $\eta$-periodic $\Sp$-oriented homotopy commutative ring spectrum $\E$ is an FTL over the ring $\E_*$ with framed involution $(\iota^{\E}, \Psi^{\E})$ induced by the framed involution $(\iota^M, \Psi^M)$ over $\MSpetainv$ (cf.~Proposition~\ref{prop:msp-framed-involution}).
\end{proposition}
\begin{proof}
	We only need to prove this for the universal GFTL $F_t^M$ of $\MSpetainv$ over $k=\ZZ$.
	Since the Hurewicz map $h:\MSpetainv_* \to \HW_*\MSp$ is injective, it suffices to check that $F_t^h:=h(F^M_t)$ is an FTL over $\HW_*\MSp$ with framed involution $(h(\iota^{M}), h(\Psi^{M}))$. By Proposition~\ref{prop:hurewicz-gftl-roots}, the roots of $F_t^h$ are
	\[
		r_{\varepsilon}(x,y,z):=\exp \left(\varepsilon_1 \log(x) + \varepsilon_2 \log(y) + \varepsilon_3 \log(z)\right),
		\qquad \varepsilon_1\varepsilon_2\varepsilon_3=-1.
	\]
	Symmetry is immediate from the symmetry of this root set. For the neutral element, the roots at $(x,0,0)$ are $x,x,h(\iota^M)(x),h(\iota^M)(x)$, hence
	\[
		F_t^h(x,0,0)=(1+xt)^2(1+h(\iota^M)(x)t)^2.
	\]
	For the weak neutral element, the roots at $(x,x,y)$ are $y,y,\exp(\pm2\log(x)-\log(y))$, so $y$ is a multiple root of $F_t^h(x,x,y)$. For framed $\iota$-linearity, we have $\widehat{F}_t:= h(\Psi^M)_*F_t^h=\vartheta_*F_t^{\HW}$ with $\vartheta(u):=h(\Psi^M)(\exp(u))=(\exp(u)-\exp(-u))/2$. The series $\vartheta$ is odd. Replacing $x$ by $-x$ changes the index set $\{\varepsilon_1\varepsilon_2\varepsilon_3=-1\}$ to its negative, hence $\widehat{F}_t(-x,y,z)=\widehat{F}_{-t}(x,y,z)$. For associativity, both $F_t^h(F_t^h(x,y,z),u,v)$ and $F_t^h(x,F_t^h(y,z,u),v)$ have roots
	\[
		\exp(\sigma_1\log(x)+\sigma_2\log(y)+\sigma_3\log(z)+\sigma_4\log(u)+\sigma_5\log(v)), \quad \sigma_i\in\{\pm 1\},\prod_{i} \sigma_i =1
	\]
	so they agree.
\end{proof}

\begin{example}
	The frame $\Psi^M$ transports the universal GFTL $F^M_t$ over $\MSpetainv$ to the normalized FTL
	\[
		\widehat{F}^M_t:=\Psi^M_*F^M_t.
	\]
	Equivalently, $\widehat{F}^M_t$ is the GFTL associated to the $\Sp$-orientation $s=\Psi^M(t)$ induced by the endomorphism $\xi$ of Theorem~\ref{thm:quillen-idempotent}. Since this orientation factors through $\MSLetainv$ and has trivial associated framed involution, we may also view $\widehat{F}^M_t$ as the induced normalized GFTL over $\MSLetainv$, by abuse of notation.
\end{example}
\begin{proposition}\label{prop:ftl-nonpositive-coefficients}
	The non-positive degree coefficients of $F_t(x,y,z)$ are completely determined by the axioms.
\end{proposition}
\begin{proof}
	First replace $F_t$ by the normalized law $\Psi_*F_t$, and write it again as $F_t$. Its $\iota$-linearity gives
	\[
		F_l^{\le 0}(-x,y,z)=(-1)^l F_l^{\le 0}(x,y,z).
	\]
	Combined with symmetry and $2$-saturation, each monomial $x^ay^bz^c$ in $F_l^{\le 0}$ must have all exponents congruent to $l$ modulo~$2$. In particular $F_1^{\le 0}=0$, and
	\[
		\begin{aligned}
			            & F_2^{\le 0}=a^2_{2,0,0}\sigma(x^2),\qquad
			F_3^{\le 0}=a^3_{1,1,1}xyz,                                                             \\
			F_4^{\le 0} & =a^4_{2,0,0}\sigma(x^2)+a^4_{4,0,0}\sigma(x^4)+a^4_{2,2,0}\sigma(x^2y^2).
		\end{aligned}
	\]

	From the neutral element axiom
	\[
		F_t(x,0,0)=(1+xt)^2(1+\iota(x)t)^2
		=1-2x^2t^2+x^4t^4+\text{(terms of higher $x$-degree)},
	\]
	we obtain
	\[
		a^2_{2,0,0}=-2,\qquad
		a^4_{2,0,0}=0,\qquad
		a^4_{4,0,0}=1.
	\]
	Next, from the weak neutral element axiom
	\[
		F_t(x,x,y)=(1+yt)^2(1+S(x,y)t+P(x,y)t^2),
	\]
	the $t$- and $t^2$-coefficients give
	\[
		F_1(x,x,y)=2y+S(x,y),\qquad
		F_2(x,x,y)=y^2+2yS(x,y)+P(x,y).
	\]
	Since $F_1^{\le 0}=0$ and $F_2^{\le 0}(x,y,z)=-2\sigma(x^2)$, we obtain
	\[
		S^{\le 0}(x,y)=-2y,\qquad
		P^{\le 0}(x,y)=-4x^2+y^2.
	\]
	The $t^3$-coefficient is
	\[
		F_3(x,x,y)=2yP(x,y)+y^2S(x,y).
	\]
	Comparing the coefficient of $x^2y$, only the term $2yP(x,y)$ contributes, so
	\[
		a^3_{1,1,1}=[F_3(x,x,y)]_{x^2y}=2[P(x,y)]_{x^2}=2(-4)=-8.
	\]
	Hence
	\[
		F_2^{\le 0}(x,y,z)=-2\sigma(x^2),\qquad
		F_3^{\le 0}(x,y,z)=-8xyz.
	\]

	It remains to determine $a^4_{2,2,0}$. From the weak neutral element axiom, the $t^4$-coefficient satisfies $F_4(x,x,y)=y^2P(x,y)$. Using the formula for $P^{\le 0}$ above and comparing the degree-$4$ part of $F_4^{\le 0}(x,x,y)=y^2P^{\le 0}(x,y)$:
	\[
		(2+a^4_{2,2,0})x^4+2a^4_{2,2,0}x^2y^2+y^4=-4x^2y^2+y^4,
	\]
	giving $a^4_{2,2,0}=-2$. Therefore
	\[
		F_4^{\le 0}(x,y,z)=\sigma(x^4)-2\sigma(x^2y^2).
	\]
	This also shows that the normalized law has no negative-degree coefficients. Since the nonlinear coefficients of $\Psi^{\pm1}$ have positive degree, transporting back preserves the displayed non-positive part.
\end{proof}

\begin{remark}
	Compared with \cite[\S A.5]{coulette2024formal}, the strengthened weak neutral element axiom forces $a^3_{1,1,1}=-8$, eliminating the ambiguity $a^3_{1,1,1}=\pm 8$.
\end{remark}

\begin{definition}\label{def:trivial-ftl}
	Let $R$ be a graded ring. The \emph{trivial FTL} over $R$ is the FTL over $R$ with the trivial framed involution, given by
	\[
		F_t^{\mathrm{triv}}(x,y,z)=1-2\sigma(x^2)t^2-8xyz t^3+\left(\sigma(x^4)-2\sigma(x^2y^2)\right)t^4.
	\]
	By Proposition~\ref{prop:ftl-nonpositive-coefficients}, this is the unique FTL whose positive-degree coefficients vanish. For example, the GFTL in Example~\ref{ex:gftl-over-hw} is trivial.
\end{definition}

\subsection{Walter ring}\label{subsec:walter-ring}
We now introduce the $\eta$-periodic Walter ring via the following universal characterization.

\begin{proposition}
	There exists a graded $\Rfr$-algebra $\W$, called the \emph{$\eta$-periodic Walter ring}, together with the framed involution $(\iota^{\univ}, \Psi^{\univ})$ induced by the structure map $\Rfr\to\W$, and an FTL $F_t^{\univ}$ over $\W$, with the following universal property: for every graded ring $R$ endowed with a framed involution $(\iota^R, \Psi^R)$ and an FTL $F_t^R$ over it, there is a unique graded ring homomorphism $g:\W\to R$ such that
	\[
		g_*\Psi^{\univ}=\Psi^R,\qquad g_*F_t^{\univ}=F_t^R.
	\]
	Equivalently, the functor $\mathcal{FTL}$ is represented by $\W$ in $\mathbf{Rings}^{\mathrm{gr}}$.
\end{proposition}

\begin{proof}
	Write
	\[
		F_t^{\univ}(x,y,z)=1 + \sum_{l=1}^{4}\left(\sum_{i,j,k\ge 0} a^l_{i,j,k} x^i y^j z^k\right)t^l,
	\]
	with grading determined by $|a^l_{i,j,k}|=2(i+j+k-l)$ (in particular, $a^l_{i,j,k}$ occurs only when $i+j+k-l\ge 0$ and is fixed when $i+j+k-l = 0$, cf.~Proposition~\ref{prop:ftl-nonpositive-coefficients}). Let
	\[
		P:=\Rfr\left[ a^l_{i,j,k} \;\middle|\; i+j+k-l\ge 1,\ i,j,k\ge 0 \right].
	\]
	Define
	\[
		\W:=P/\mathcal{I}^{\mathrm{sat}(2)},
	\]
	where $\mathcal{I}$ is the homogeneous ideal generated by the relations expressing that $F_t^{\univ}$ is an FTL with framed involution induced from $\Rfr$, and $\mathcal{I}^{\mathrm{sat}(2)}$ is its $2$-saturated closure in the sense of Definition~\ref{def:2-saturated-relations}.
	Given $(\iota^R, \Psi^R, F_t^R)\in \mathcal{FTL}(R)$, sending the generators $b_m$ and $a^l_{i,j,k}$ to the corresponding coefficients of $\Psi^R$ and $F_t^R$ defines a graded ring map $g:\W\to R$, and by construction this $g$ is unique.
\end{proof}

\begin{remark}\label{rem:walter-ring-splitting}
	In particular, $\widetilde{\W}:=\W \otimes_{\Rfr} \ZZ$, where $\Rfr \to \ZZ$ is the map classifying the trivial involution, represents the functor $\mathcal{FTL}_{\mathrm{norm}}$ of normalized FTLs. Moreover, any FTL is obtained by transporting a normalized FTL along a framed involution. Hence the representing ring splits as
	\[
		\W \cong \Rfr \otimes_{\ZZ} \widetilde{\W}.
	\]
\end{remark}

\subsection{From FGL to FTL}\label{subsec:from-fgl-to-ftl}

Suppose that $R$ is a graded $\ZZ[1/2]$-algebra and $\mu$ is a formal group law over $R$. Then the formal inverse $\iota(x) \in R \llbracket x \rrbracket$ of $\mu$ is an involution over $R$, hence a framed involution with
\[
	\Psi(x):=\frac{x-\iota(x)}{2}.
\]
Formally defining $x^{-1}:=\iota(x)$, we can then define a formal ternary law over the framed involution $(\iota, \Psi)$ as follows:
\[
	F_t(x,y,z):=\prod_{\substack{\varepsilon_1,\varepsilon_2,\varepsilon_3\in\{\pm 1\}\\ \varepsilon_1\varepsilon_2\varepsilon_3=-1}} \left( 1 +  \mu(\mu(x^{\varepsilon_1},y^{\varepsilon_2}),z^{\varepsilon_3}) t\right).
\]
It satisfies the symmetry and associativity axioms by the symmetry and associativity of $\mu$. Since $\mu(x,x^{-1}) = 0$ and $\mu(x,0) = x$, it also satisfies the neutral element and weak neutral element axioms. For framed $\iota$-linearity, write $\nu:=\Psi_*\mu$ and let $\widehat{F}_t:=\Psi_*F_t$, so that $\widehat{F}_t$ is given by the same product formula with $\mu$ replaced by $\nu$. Then $\nu$ has formal inverse $-x$. Replacing $x$ by $-x$ in the defining product for $\widehat{F}_t$ changes the index condition from $\varepsilon_1\varepsilon_2\varepsilon_3=-1$ to $\varepsilon_1\varepsilon_2\varepsilon_3=1$, while changing all three signs identifies the latter index set with the negatives of the former. Hence $\widehat{F}_t(-x,y,z)=\widehat{F}_{-t}(x,y,z)$.

This construction is compatible with base change in the category $\mathbf{Alg}^{\mathrm{gr}}_{\ZZ[1/2]}$ of graded $\ZZ[1/2]$-algebras, so it defines a natural transformation
\[
	\mathbb{T}: \mathcal{FGL}(-) \to \mathcal{FTL}(-).
\]
over $\mathbf{Alg}^{\mathrm{gr}}_{\ZZ[1/2]}$.

\begin{remark}\label{rem:real-realization-of-T}
	The transformation $\mathbb{T}$ also has a real-realization interpretation, suggested by Alexey Ananyevskiy. When $\mathrm{R}\operatorname{Spec}(k)$ is a point (for example, $k=\ZZ$ or $\mathbb{R}$), the real realization identifies $\MSp[\etainv,1/2]$ with $\MU[1/2]$. Since $\Sp_2(\mathbb{R})\cong \mathrm{SL}_2(\mathbb{R})$ deformation retracts onto $\mathrm{SO}(2)\cong \mathrm{U}(1)$, the realization of a rank-$2$ symplectic bundle may be regarded homotopically as a complex line bundle. Let $L_i$ correspond to $\mathfrak{U}_i$, with complex structure $J_i$. The realization of the triple tensor product is
	\[
		V:=(L_1)_{\mathbb{R}}\otimes_{\mathbb{R}}(L_2)_{\mathbb{R}}\otimes_{\mathbb{R}}(L_3)_{\mathbb{R}},
	\]
	with complex structure $J_1\otimes J_2\otimes J_3$. As a complex rank-$4$ bundle, it decomposes as
	\[
		(V,J_1\otimes J_2\otimes J_3)\cong
		\bigoplus_{\substack{\varepsilon_i\in\{\pm1\}\\ \varepsilon_1\varepsilon_2\varepsilon_3=-1}}
		L_1^{\varepsilon_1}\otimes_{\mathbb{C}} L_2^{\varepsilon_2}\otimes_{\mathbb{C}} L_3^{\varepsilon_3},
		\qquad L_i^{-1}:=L_i^\vee.
	\]
	For a complex orientation with formal group law $\mu$ and $x_i=c_1(L_i)$, its four Chern roots are $\mu(\mu(x_1^{\varepsilon_1},x_2^{\varepsilon_2}),x_3^{\varepsilon_3})$, indexed by $\varepsilon_1\varepsilon_2\varepsilon_3=-1$. These are exactly the roots defining $\mathbb{T}(\mu)$.
\end{remark}

\begin{example}[Multiplicative formal group law]\label{ex:multiplicative-fgl-kw}
	Let $R=\ZZ[1/2,\alpha]$, with $|\alpha|=2$, and over it we have the multiplicative formal group law
	\[
		\mu_m(x,y)=x+y+\alpha xy.
	\]
	Its formal inverse and frame are
	\[
		\iota(x)=-\frac{x}{1+\alpha x},\qquad
		\Psi(x)=\frac{x-\iota(x)}{2}=\frac{x(2+\alpha x)}{2(1+\alpha x)}.
	\]
	Transporting $\mu_m$ along $\Psi$ gives
	\[
		\Psi_*\mu_m(x,y)
		=x\sqrt{1+\alpha^2y^2}+y\sqrt{1+\alpha^2x^2},
	\]
	where the square roots are expanded as formal power series. This is the specialization $\delta=-\alpha^2/2$, $\varepsilon=0$ of the elliptic-sine formal group law in \cite[Example~20]{buchstaber2010elliptic}; see also \cite{koyama2013formal}. Note that the construction $\mathbb{T}$ commutes with transport, and then we have $\Psi_*\mathbb{T}(\mu_m)=\mathbb{T}(\Psi_*\mu_m)$.
	Write $\widehat{F}_t:=\mathbb{T}(\Psi_*\mu_m)$, a direct expansion gives
	\[
		\begin{aligned}
			\widehat{F}_{1}(x,y,z) & = -4\alpha^2 xyz,                                        \\
			\widehat{F}_{2}(x,y,z) & = -4\alpha^2 \sigma(x^2y^2) - 2\sigma(x^2),              \\
			\widehat{F}_{3}(x,y,z) & = -4\alpha^2 \sigma(x^3yz)-8xyz,                         \\
			\widehat{F}_{4}(x,y,z) & = -4\alpha^2 x^2y^2z^2 + \sigma(x^4) - 2\sigma(x^2 y^2).
		\end{aligned}
	\]
	Here the transported ternary law agrees with the GFTL of $\KW$ in Example~\ref{ex:gftl-over-kw}, after artificially identifying this $-4\alpha^2$ with the Bott element $\beta$.
\end{example}

\subsection{From FTL to FGL}\label{subsec:from-ftl-to-fgl}
We now define a natural transformation
\[
	\mathbb{G}: \mathcal{FTL}(-) \to \mathcal{FGL}(-),
\]
over $\mathbf{Alg}^{\mathrm{gr}}_{\ZZ[1/2]}$. Note that the invertibility of $2$ is vital here.

\begin{remark}[Square root] In the formal power series ring $S:= R \llbracket x_1, x_2, \ldots, x_{n} \rrbracket$ over a $\ZZ[1/2]$-algebra $R$, suppose that $u \in S_{+}:= \ker(S \to R)$. Then we can define the following formal square root:
	\[
		\sqrt{1+u}=(1+u)^{1/2}=\sum_{n\ge 0}\binom{1/2}{n}u^n,
		\qquad
		\binom{1/2}{n}=\frac{\left(\frac12\right)\left(\frac12-1\right)\cdots\left(\frac12-n+1\right)}{n!}.
	\]
	Moreover, if $w = (1+u)v^{2}$ for $v \in S$, there are two choices of formal square root:
	\[
		\sqrt{w} := \pm v\sqrt{1+u}.
	\]
\end{remark}

\begin{lemma}\label{lem:ftl-xxy-splitting}
	Given an FTL $F_t$ over a graded $\ZZ[1/2]$-algebra with framed involution $(\iota, \Psi)$, the $(4,2)$-series $F_t(x,x,y)$ splits:
	\[
		F_t(x,x,y)=(1+yt)^2(1+B(x,y)t)(1+B'(x,y)t),
	\]
	with
	\[
		B(x,y) \equiv - 2x - y \bmod (x,y)^2, \quad B'(x,y) \equiv 2x - y \bmod (x,y)^2.
	\]
\end{lemma}
\begin{proof}
	Define
	\[
		S(x,y) := F_1(x,x,y)-2y, \qquad P(x,y) := F_2(x,x,y)-2yS(x,y)-y^2.
	\]
	These are the series appearing in the weak neutral element axiom, and one has
	\[
		D(x,y):=S^2(x,y)-4P(x,y)\equiv 16x^2 \bmod (x,y)^3.
	\]
	In fact $x^2\mid D(x,y)$. This is proved in Lemma~\ref{lem:d-divisible-by-x-squared}.

	Choose $\sqrt{D(x,y)}\in R\llbracket x,y\rrbracket$ with
	\[
		\sqrt{D(x,y)}\equiv 4x \pmod{(x,y)^2},
	\]
	and define
	\[
		B(x,y):=\frac{S(x,y)-\sqrt{D(x,y)}}{2},\qquad
		B'(x,y):=\frac{S(x,y)+\sqrt{D(x,y)}}{2}.
	\]
	These are exactly the two roots of the power series $1+ S(x,y)t + P(x,y)t^2$. Equivalently, they give the two nontrivial roots of $F_t(x,x,y)$ in addition to the double root $y$.
\end{proof}

\begin{lemma}\label{lem:d-divisible-by-x-squared}
	$D(x,y)$ is divisible by $x^2$.
\end{lemma}
\begin{proof}
	Write
	\[
		D(x,y)=\sum_{n\ge 0} d_n(y)x^n,\qquad d_n(y)\in R\llbracket y\rrbracket.
	\]
	It suffices to prove $d_0(y)=d_1(y)=0$.

	By symmetry and the neutral element axiom,
	\[
		F_t(0,0,y)=F_t(y,0,0)=(1+yt)^2(1+\iota(y)t)^2.
	\]
	Therefore $F_1(0,0,y)=2y+2\iota(y)$ and
	$F_2(0,0,y)=y^2+4y\iota(y)+\iota(y)^2$. By the definitions of $S$ and $P$ this gives
	\[
		S(0,y)=2\iota(y),\qquad P(0,y)=\iota(y)^2.
	\]
	Hence $D(0,y)=S(0,y)^2-4P(0,y)=0$, i.e.\ $d_0(y)=0$.

	By Remark~\ref{rem:epsilon-linearity-consequence}, $F_t(x,x,y)=F_t(\iota(x),\iota(x),y)$ and thus
	$D(x,y)=D(\iota(x),y)$. Using $\iota(x)=-x+(\text{terms in }x^2)$, we get
	\[
		D(\iota(x),y)=d_1(y)\iota(x)+\cdots=-d_1(y)x+(\text{terms in }x^2).
	\]
	Comparing coefficients of $x$ in $D(x,y)=D(\iota(x),y)$ gives $d_1(y)=-d_1(y)$,
	so $2d_1(y)=0$. As $2$ is invertible in $R$, we conclude $d_1(y)=0$.
\end{proof}

\begin{definition}
	For a given FTL $F_t$ over a graded $\ZZ[1/2]$-algebra with framed involution $(\iota, \Psi)$, let $B(x,y)$ be the root produced in Lemma~\ref{lem:ftl-xxy-splitting}.
	Set
	\[
		f(x,y):=B(\iota(x),\iota(y)).
	\]
	Moreover, the series
	\[
		[2](x):=f(x,0)\equiv 2x \bmod (x^2)
	\]
	has a compositional inverse $[1/2](x)$.

	We then construct the following binary series over $R$ associated to $(\iota, \Psi, F_t)$:
	\begin{equation}\label{eq:ftl-to-fgl}
		\mu(x,y):=f([1/2](x),y)
	\end{equation}
	By Proposition~\ref{prop:ftl-to-fgl-is-fgl} below, it is indeed an FGL, and therefore we obtain the desired map
	\[
		\mathbb{G}: \mathcal{FTL}(R) \to \mathcal{FGL}(R).
	\]
	This map is compatible with base change, and therefore $\mathbb{G}: \mathcal{FTL}(-) \to \mathcal{FGL}(-)$ is a natural transformation over $\mathbf{Alg}^{\mathrm{gr}}_{\ZZ[1/2]}$.
\end{definition}

\begin{proposition}\label{prop:ftl-to-fgl-is-fgl}
	For any FTL $(\iota, \Psi, F_t)$ over a graded $\ZZ[1/2]$-algebra $R$, the series
	$\mu(x,y)\in R\llbracket x,y\rrbracket$ defined in~\eqref{eq:ftl-to-fgl} is a formal group law over $R$.
\end{proposition}
\begin{proof}
	By definition, $\mu(x,0)=f([1/2](x),0)=[2]([1/2](x))=x$.
	By symmetry of $F_t$ and the neutral element axiom, $F_t(0,0,y)=(1+yt)^2(1+\iota(y)t)^2$. Thus $B(0,y)=\iota(y)$, and
	\[
		\mu(0,y)=f(0,y)=B(0,\iota(y))=\iota(\iota(y))=y.
	\]

	By Remark~\ref{rem:epsilon-linearity-consequence}, one has $F_t(x,x,y)=F_t(\iota(x),\iota(x),y)$, and comparing
	the factors in the splitting by their linear terms gives
	\begin{equation}\label{eq:swap-root-b}
		B'(x,y)=B(\iota(x),y)
	\end{equation}

	Next, apply associativity and symmetry to get an equality of split $(16,3)$-series
	\[
		F_t(x,x,F_t(y,y,z))=F_t(y,y,F_t(x,x,z)).
	\]
	Comparing the uniquely determined factors by their linear terms yields
	\[
		B(x,B(y,z))=B'(y,B'(x,z)).
	\]
	Using \eqref{eq:swap-root-b}, this can be rewritten as
	\begin{equation}\label{eq:b-associativity}
		B(x,B(y,z))=B(\iota(y),B(\iota(x),z)).
	\end{equation}

	Setting $x=0$ in \eqref{eq:b-associativity} and using $B(0,w)=\iota(w)$ and $\iota(0)=0$ gives
	\[
		\iota(B(y,z))=B(\iota(y),\iota(z))=f(y,z).
	\]
	In particular, applying this with $(y,z)=(\iota(v),\iota(w))$ yields
	\[
		\iota(f(v,w))=\iota(B(\iota(v),\iota(w)))=B(v,w).
	\]
	Therefore, from~\eqref{eq:b-associativity} we deduce
	\begin{equation}\label{eq:f-associativity}
		\begin{aligned}
			f(u,f(v,w)) & =B(\iota(u),\iota(f(v,w)))=B(\iota(u),B(v,w)) \\
			            & =B(\iota(v),B(u,w))=f(v,f(u,w)).
		\end{aligned}
	\end{equation}

	Taking $(u,v,w)=([1/2](x),[1/2](y),0)$ in~\eqref{eq:f-associativity} gives $\mu(x,y)=\mu(y,x)$.
	Taking $(u,v,w)=([1/2](x),[1/2](y),z)$ gives
	\[
		\mu(x,\mu(y,z))=\mu(y,\mu(x,z)).
	\]
	Using commutativity, for any $x,y,z$,
	\[
		\mu(\mu(x,y),z)=\mu(z,\mu(x,y))=\mu(x,\mu(z,y))=\mu(x,\mu(y,z)),
	\]
	so $\mu$ is associative. Hence $\mu$ is a formal group law.
\end{proof}

\begin{corollary}\label{cor:explicit-ftl-splitting-from-G}
	Let $(\iota,\Psi,F_t)$ be an FTL over a graded $\ZZ[1/2]$-algebra $R$, let $\mu=\mathbb{G}(\iota,\Psi,F_t)$, and set
	\[
		p:=\mu(x,y),\qquad q:=\mu(x,\iota(y)),\qquad s:=[1/2](p),\qquad r:=[1/2](q).
	\]
	Then $F_t(x,y,z)$ splits with roots:
	\[
		B(s,z), \quad B(\iota(s),z), \quad B(r,\iota(z)), \quad B(\iota(r),\iota(z)).
	\]
	Moreover, $\iota(x)$ is the formal inverse of $\mu$, and the series $[2](x)=f(x,0)$ is the $2$-series of $\mu$, that is $[2](x)=[2]_{\mu}(x):=\mu(x,x)$.
\end{corollary}
\begin{proof}
	Lemma~\ref{lem:ftl-xxy-splitting} and equation~\eqref{eq:swap-root-b} give
	\[
		F_t(a,a,b)=(1+bt)^2(1+B(a,b)t)(1+B(\iota(a),b)t).
	\]
	Applying associativity to this split form and using $[2](s)=p$, $[2](r)=q$ isolates the four factors of $F_t(x,y,z)$. Their linear terms are
	\[
		-x-y-z,\qquad x+y-z,\qquad -x+y+z,\qquad x-y+z,
	\]
	which are the linear terms of $B(s,z)$, $B(\iota(s),z)$, $B(r,\iota(z))$, and $B(\iota(r),\iota(z))$, respectively. Since the split factors are distinguished by their linear terms, this gives the displayed product.

	For the remaining assertions, first note the two identities from the proof of Proposition~\ref{prop:ftl-to-fgl-is-fgl}:
	\[
		\iota(f(v,w))=f(\iota(v),\iota(w)),\qquad B(v,w)=f(\iota(v),\iota(w)).
	\]
	In particular,
	\[
		[2](\iota(v))=\iota([2](v)),\qquad [1/2](\iota(v))=\iota([1/2](v)),
	\]
	and
	\[
		B(a,0)=\iota([2](a)),\qquad B(\iota(a),0)=[2](a).
	\]

	Set $y=x$ and $z=0$ in the displayed product, so that $p=\mu(x,x)$ and $q=\mu(x,\iota(x))$. If $h:=[1/2](x)$, then
	\[
		\iota(q)=\iota(f(h,\iota(x)))=f(\iota(h),x)=\mu(\iota(x),x)=q,
	\]
	where the last equality uses commutativity of $\mu$. The displayed product gives
	\[
		F_t(x,x,0)=(1+\iota(p)t)(1+pt)(1+\iota(q)t)(1+qt).
	\]
	The same split form with $(a,b)=(x,0)$ gives
	\[
		F_t(x,x,0)=(1+\iota([2](x))t)(1+[2](x)t).
	\]
	Comparing the factors with linear terms $-2x$ and $2x$ gives $p=[2](x)$. After cancellation,
	\[
		(1+\iota(q)t)(1+qt)=1.
	\]
	Since $q=\iota(q)$, the coefficient of $t$ gives $2q=0$, hence $q=0$. Therefore $\mu(x,\iota(x))=0$, so $\iota(x)$ is the formal inverse for $\mu$, and $[2](x)=\mu(x,x)$ is its $2$-series.
\end{proof}

\begin{proposition}\label{prop:G-inverse-of-T}
	Suppose that $(\iota, \Psi, F_t)=\mathbb{T}(\mu_0(x,y))$. Then $\mathbb{G}(\iota, \Psi, F_t)=\mu_0$.
\end{proposition}
\begin{proof}
	Let $+_0$ denote the addition for $\mu_0$, $\bar{a}:=\iota_{\mu_0}(a)$, and $u:=[2]_{\mu_{0}}(x)$. By the definition of $\mathbb{T}$, the roots of the ternary law $F_t(x,x,y)$ are $\{y,y,u+_0\bar{y},\bar{u}+_0\bar{y}\}$. We compute the reconstruction series derived from the coefficients $F_1,F_2$:
	\[
		\begin{aligned}
			S(x,y)   & :=F_1(x,x,y)-2y=(u+_0\bar{y})+(\bar{u}+_0\bar{y}),                       \\
			P(x,y)   & :=F_2(x,x,y)-2yS(x,y)-y^2=(u+_0\bar{y})(\bar{u}+_0\bar{y}),              \\
			D(x,y)   & :=S^2-4P=((u+_0\bar{y})-(\bar{u}+_0\bar{y}))^2,                          \\
			B(x,y)   & :=\tfrac12(S-\sqrt{D})=\bar{u}+_0\bar{y}=[-2]_{\mu_{0}}(x)+_0\bar{y},    \\
			f(x,y)   & :=B(\bar{x},\bar{y})=[-2]_{\mu_{0}}(\bar{x})+_0 y=[2]_{\mu_{0}}(x)+_0 y, \\
			[2](x)   & :=f(x,0)=[2]_{\mu_{0}}(x) \Longrightarrow [1/2](x)=[1/2]_{\mu_{0}}(x),   \\
			\mu(x,y) & :=f([1/2](x),y)=[2]_{\mu_{0}}([1/2]_{\mu_{0}}(x))+_0 y=\mu_0(x,y).
		\end{aligned}
	\]
\end{proof}

\begin{proposition}\label{prop:T-inverse-of-G}
	Suppose that $(\iota,\Psi,F_t)$ is an FTL over a graded $\ZZ[1/2]$-algebra $R$, and set $\mu:=\mathbb{G}(\iota,\Psi,F_t)$. Then $\mathbb{T}(\mu)=(\iota,\Psi,F_t)$.
\end{proposition}
\begin{proof}
	By Corollary~\ref{cor:explicit-ftl-splitting-from-G}, the formal inverse of $\mu$ is $\iota$, and $[2](a) = [2]_{\mu}(a)$. Hence $\iota(\mu(a,b))=\mu(\iota(a),\iota(b))$ and $[2](\iota(a))=\iota([2](a))$.

	With $p,q,s,r$ as in Corollary~\ref{cor:explicit-ftl-splitting-from-G}, the explicit splitting gives
	\[
		F_t(x,y,z)=(1+B(s,z)t)(1+B(\iota(s),z)t)(1+B(r,\iota(z))t)(1+B(\iota(r),\iota(z))t).
	\]
	Since $B(a,b)=f(\iota(a),\iota(b))=\mu([2](\iota(a)),\iota(b))$, these four roots are
	\[
		\begin{aligned}
			B(s,z)               & =\mu(\iota(p),\iota(z))=\mu(\mu(\iota(x),\iota(y)),\iota(z)), \\
			B(\iota(s),z)        & =\mu(p,\iota(z))=\mu(\mu(x,y),\iota(z)),                      \\
			B(r,\iota(z))        & =\mu(\iota(q),z)=\mu(\mu(\iota(x),y),z),                      \\
			B(\iota(r),\iota(z)) & =\mu(q,z)=\mu(\mu(x,\iota(y)),z).
		\end{aligned}
	\]
	They are exactly the roots defining $\mathbb{T}(\mu)$. Thus $\mathbb{T}(\mu)=F_t$.
\end{proof}

\subsection{Lazard- and Quillen-type theorems}
Propositions~\ref{prop:G-inverse-of-T} and~\ref{prop:T-inverse-of-G} show that the maps $\mathbb{T}$ and $\mathbb{G}$ are inverse bijections between $\mathcal{FGL}(R)$ and $\mathcal{FTL}(R)$ for any graded $\ZZ[1/2]$-algebra $R$. Hence, by the representability of the restrictions of $\mathcal{FGL}$ and $\mathcal{FTL}$ to $\mathbf{Alg}^{\mathrm{gr}}_{\ZZ[1/2]}$, we obtain the following Lazard-type theorem.

\begin{theorem}[Lazard-type]\label{thm:lazard-walter}
	The functor $\mathbb{T}$ induces an isomorphism
	\[
		\W[1/2] \xrightarrow{\cong} L[1/2]
	\]
	as graded rings.
\end{theorem}

As a consequence, we have the following Quillen-type theorem.
\begin{theorem}[Quillen-type]\label{thm:quillen-walter}
	Suppose that $W(k)\cong \ZZ$ (e.g.\ $k=\ZZ$ or $\mathbb{R}$). Then the universal GFTL $F^{M}_t$ and the framed involution $(\iota^{M}, \Psi^{M})$ determine a classifying map
	\[
		\phi:\W\to (\MSpetainv)_*.
	\]
	This map is injective. After inverting $2$, the induced map
	\[
		\phi':=\phi[1/2]: \W[1/2] \to (\MSpetainv)_*[1/2]
	\]
	is an isomorphism. Moreover, the following diagram commutes:
	\[
		\xymatrix{
		\W[1/2] \ar[r]^-{\phi'} \ar[d]_-{\theta \circ \mathbb{T}} &  (\MSpetainv)_*[1/2] \ar[d]_-{h}\\
		\ZZ[1/2][a_1, a_2, \ldots] \ar[r]^-{J} & \HW_*\MSp[1/2]
		}
	\]
	Here $h$ is the Hurewicz map and $J: \ZZ[1/2][a_1, a_2, \ldots] \to \HW_*\MSp[1/2]$ is the structural algebra isomorphism.
\end{theorem}
\begin{proof}
	It suffices to prove the claim for $k=\ZZ$. Since the two compositions $h\circ \phi'$ and $J \circ \theta \circ \mathbb{T}$ classify the same FTL, the universal property implies that the diagram commutes. As $\mathbb{T}$ is an isomorphism and, by the proof of Theorem~\ref{thm:lazard-structure} and Remark~\ref{rem:lambda-image-after-inverting-two}, the maps $Q_{2n}(\theta \circ \mathbb{T})$ and $Q_{2n}h$ are isomorphisms onto rank-one images that correspond under the isomorphism $Q_{2n}J$, it follows from the commutative diagram that $Q_{2n}\phi'$ is an isomorphism and hence that $\phi'$ is surjective. Since $\theta \circ \mathbb{T}$ is injective and $J$ is an isomorphism, $\phi'$ is also injective. Thus $\phi'$ is an isomorphism. Finally, $\W$ is $2$-torsion-free by the $2$-saturation construction, so the injectivity of $\phi'$ implies the injectivity of $\phi$.
\end{proof}

\begin{corollary}\label{cor:normalized-quillen-walter}
	Suppose that $W(k)\cong \ZZ$. Then the normalized GFTL $\widehat{F}^M_t$ over $\MSLetainv$ determines a classifying map
	\[
		\widetilde{\phi}:\widetilde{\W}\to (\MSLetainv)_*.
	\]
	This map is injective. After inverting $2$, it induces an isomorphism
	\[
		\widetilde{\phi}':=\widetilde{\phi}[1/2]:\widetilde{\W}[1/2]\to (\MSLetainv)_*[1/2].
	\]
	Moreover, under the splittings of Lemma~\ref{lem:canonical-msp-splitting} and Remark~\ref{rem:walter-ring-splitting}, one has a commutative diagram
	\[
		\xymatrix{
		\Rfr\otimes_{\ZZ}\widetilde{\W} \ar[r]^-{\cong} \ar[d]_-{\mathrm{id}\otimes \widetilde{\phi}} & \W \ar[d]^-{\phi}\\
		\Rfr\otimes_{\ZZ}(\MSLetainv)_* \ar[r]^-{\chi} & (\MSpetainv)_*.
		}
	\]
\end{corollary}
\begin{proof}
	Under the identifications in Remark~\ref{rem:walter-ring-splitting} and Lemma~\ref{lem:canonical-msp-splitting}, the composite $\chi\circ (\mathrm{id}\otimes \widetilde{\phi})$ classifies the same FTL over $(\MSpetainv)_*$ as $\phi$, namely the universal GFTL $F_t^M$. Hence the square commutes by the universal property of $\W$. After inverting $2$, the localized square shows that $\widetilde{\phi}'$ is an isomorphism, because the top horizontal map, the bottom horizontal map, and $\phi'$ are isomorphisms. Finally, $\widetilde{\W}$ is $2$-torsion-free, so the injectivity of $\widetilde{\phi}'$ implies the injectivity of $\widetilde{\phi}$.
\end{proof}

\appendix

\section{Computations}\label{app:gftl-hurewicz-computations}
Throughout this appendix, congruences $\equiv$ are taken modulo decomposables in the coefficient ring.
\subsection{Proof of Proposition~\ref{prop:hurewicz-indecomposable-formulas}}

\begin{proof}[Proof of Proposition~\ref{prop:hurewicz-indecomposable-formulas}]
	The formula for $h(F_1)$ is from Proposition~\ref{prop:hurewicz-gftl-roots}. For the remaining coefficients, write $x^H,y^H,z^H$ for the generators induced by the $\Sp$-orientation of $\HW$, so that
	\[
		x^H=\log(x),\qquad y^H=\log(y),\qquad z^H=\log(z).
	\]
	Let
	\[
		t_1=x^H+y^H-z^H,\; t_2=x^H-y^H+z^H,\; t_3=-x^H+y^H+z^H,\; t_4=-x^H-y^H-z^H.
	\]
	These are the four roots of $F_t^{\HW}$.

	We compute the three higher coefficients in turn.

	For $F_2$, Proposition~\ref{prop:hurewicz-gftl-roots} gives
	\[
		h(F_2)=\sum_{i<j}\exp(t_i)\exp(t_j).
	\]
	Modulo decomposables,
	\[
		\exp(u)\exp(v)\equiv uv+\sum_{n=1}^{\infty}a_n\left(u^{n+1}v+uv^{n+1}\right).
	\]
	Since
	\[
		\begin{aligned}
			\sum_{i<j}t_it_j
			 & =-2\left((x^H)^2+(y^H)^2+(z^H)^2\right), \\
			\sum_{i<j}\left(t_i^{n+1}t_j+t_it_j^{n+1}\right)
			 & =-S_{n+2}(x^H,y^H,z^H),
		\end{aligned}
	\]
	we obtain
	\[
		\begin{aligned}
			h(F_2)
			 & \equiv -2\left((x^H)^2+(y^H)^2+(z^H)^2\right)-\sum_{n=1}^{\infty}a_n S_{n+2}(x^H,y^H,z^H)               \\
			 & = -2\left(\log^2(x)+\log^2(y)+\log^2(z)\right)-\sum_{n=1}^{\infty}a_n S_{n+2}(\log(x),\log(y),\log(z)).
		\end{aligned}
	\]
	Moreover,
	\[
		\log^2(u)\equiv \left(u-\sum_{n=1}^{\infty}a_nu^{n+1}\right)^2
		\equiv u^2-2\sum_{n=1}^{\infty}a_nu^{n+2}.
	\]
	Putting these together yields
	\[
		h(F_2)\equiv -2Q(x,y,z)
		+\sum_{n=1}^{\infty}a_n P_{n+2}(x,y,z).
	\]

	For $F_3$, we similarly start from
	\[
		h(F_3)=\sum_{i<j<k}\exp(t_i)\exp(t_j)\exp(t_k).
	\]
	Working modulo decomposables, we obtain
	\[
		\begin{aligned}
			h(F_3)
			 & =\sum_{i<j<k}\left(t_i+\sum_{n\ge 1}a_nt_i^{n+1}\right)
			\left(t_j+\sum_{n\ge 1}a_nt_j^{n+1}\right)
			\left(t_k+\sum_{n\ge 1}a_nt_k^{n+1}\right)                 \\
			 & \equiv \sum_{i<j<k}t_it_jt_k
			+\sum_{n\ge 1}a_n\sum_{i<j<k}\left(t_i^{n+1}t_jt_k+t_it_j^{n+1}t_k+t_it_jt_k^{n+1}\right).
		\end{aligned}
	\]
	A standard symmetric-polynomial identity gives
	\[
		\sum_{i<j<k}\left(t_i^{n+1}t_jt_k+t_it_j^{n+1}t_k+t_it_jt_k^{n+1}\right)
		=
		\sigma_2(t_1,\ldots,t_4)\sum_{r=1}^4 t_r^{n+1}+\sum_{r=1}^4 t_r^{n+3}.
	\]
	Therefore
	\[
		\begin{aligned}
			 & h(F_3)                                                                                                \\
			 & \equiv -8x^{H}y^{H}z^{H}
			+\sum_{n\ge 1}a_n\left(S_{n+3}(x^{H},y^{H},z^{H})-2Q(x^{H},y^{H},z^{H})S_{n+1}(x^{H},y^{H},z^{H})\right) \\
			 & \equiv -8\log(x)\log(y)\log(z)
			+\sum_{n\ge 1}a_n\left(S_{n+3}(x,y,z)-2Q(x,y,z)S_{n+1}(x,y,z)\right).
		\end{aligned}
	\]
	Finally,
	\[
		-8\log(x)\log(y)\log(z)
		\equiv -8xyz+8\sum_{n\ge 1}a_n\left(x^{n+1}yz+xy^{n+1}z+xyz^{n+1}\right),
	\]
	so
	\[
		h(F_3)\equiv -8xyz+\sum_{n=1}^{\infty} a_n C_n(x,y,z).
	\]

	Finally, for $F_4$ we start from
	\[
		h(F_4)
		=\prod_{\varepsilon_1\varepsilon_2\varepsilon_3=-1}
		\exp \left(\varepsilon_1\log(x)+\varepsilon_2\log(y)+\varepsilon_3\log(z)\right).
	\]
	Modulo decomposables,
	\[
		\begin{aligned}
			h(F_4)
			 & =\prod_{\varepsilon_1\varepsilon_2\varepsilon_3=-1}
			\left(
			(\varepsilon_1x^H+\varepsilon_2y^H+\varepsilon_3z^H)
			+\sum_{n=1}^{\infty}a_n(\varepsilon_1x^H+\varepsilon_2y^H+\varepsilon_3z^H)^{n+1}
			\right)                                                                                  \\
			 & \equiv \Delta(x^H,y^H,z^H)\left(1+\sum_{n=1}^{\infty}a_nS_n(x^H,y^H,z^H)\right)       \\
			 & \equiv \Delta(\log(x),\log(y),\log(z))\left(1+\sum_{n=1}^{\infty}a_nS_n(x,y,z)\right) \\
			 & \equiv \Delta(\log(x),\log(y),\log(z))+\Delta(x,y,z)\sum_{n=1}^{\infty}a_nS_n(x,y,z).
		\end{aligned}
	\]
	Hence
	\[
		h(F_4)\equiv \Delta(x,y,z)+\sum_{n=1}^{\infty}a_n\left(\Delta(x,y,z)S_n(x,y,z)-4L_n(x,y,z)\right).
	\]
\end{proof}

\subsection{Proof of Theorem~\ref{thm:lambda-one-hurewicz-image}}

\begin{proof}[Proof of Theorem~\ref{thm:lambda-one-hurewicz-image}]
	Modulo decomposables, $m_i \equiv -a_i$, so the coefficient of $a_n$ in $h(F_1)$ is the polynomial $S_{n+1}(x,y,z)$. It therefore suffices to compute, for $m\ge 2$, the content
	\[
		g_m:= \gcd(\mathcal{C}_m),
	\]
	where $\mathcal{C}_m$ denotes the set of non-zero coefficients of $S_m(x,y,z)$. We claim that
	\[
		g_m = \begin{cases}
			p 2^{2+\nu_2(m-1)} , & m = p^{k}, \text{ for odd prime }p, \\
			2^{2+\nu_2(m-1)},    & \text{otherwise.}
		\end{cases}
	\]

	The argument has three steps. We first rewrite $S_m$ as
	\[
		\begin{aligned}
			S_m(x,y,z) & = \sum_{\varepsilon_1 \varepsilon_2 \varepsilon_3 = -1} \sum_{a+b+c=m} \binom{m}{a,b,c} \varepsilon_1^a \varepsilon_2^b \varepsilon_3^c x^{a} y^{b} z^{c}               \\
			           & = \sum_{a+b+c=m} \left(\sum_{\varepsilon_1 \varepsilon_2 \varepsilon_3 = -1} \varepsilon_1^a \varepsilon_2^b \varepsilon_3^c\right) \binom{m}{a,b,c} x^{a} y^{b} z^{c}.
		\end{aligned}
	\]
	Define
	\[
		E(a,b,c) := \sum_{\varepsilon_1 \varepsilon_2 \varepsilon_3 = -1} \varepsilon_1^a \varepsilon_2^b \varepsilon_3^c.
	\]
	A direct check gives
	\[
		E(a,b,c)=
		\begin{cases}
			4,  & \text{if } a,b,c \text{ are all even}, \\
			-4, & \text{if } a,b,c \text{ are all odd},  \\
			0,  & \text{otherwise}.
		\end{cases}
	\]
	Equivalently,
	\[
		\mathcal{C}_m = \left\{ E(a,b,c)\binom{m}{a,b,c} \,\middle|\, a+b+c=m,\ E(a,b,c)\neq 0\right\}.
	\]

	Step 1: the even case. If $m>1$ is even, then $E(a,b,c)\neq 0$ forces $a,b,c$ all to be even, so every non-zero coefficient of $S_m$ is divisible by $4$. Since
	\[
		\binom{m}{m,0,0}=1,
	\]
	we obtain $g_m=4$.

	Step 2: the $2$-adic valuation in the odd case. Assume that $m>1$ is odd. Using Legendre's formula \cite[p.10]{legendre1830theorie}, $\nu_2(k!) = k-s_2(k)$, where $s_2(k)$ is the sum of the base-$2$ digits of $k$, we obtain for $a+b+c=m$:
	\[
		\begin{aligned}
			\nu_2 \left(\frac{m!}{a!b!c!}\right) & = \nu_2(m!) - \nu_2(a!) - \nu_2(b!) - \nu_2(c!) \\
			                                     & = s_2(a) + s_2(b) + s_2(c) - s_2(m).
		\end{aligned}
	\]
	When $E(a,b,c)\neq 0$, the integers $a,b,c$ are all odd, so we may write
	\[
		m = 2r+1,\qquad a = 2 \alpha +1,\qquad b = 2\beta +1,\qquad c = 2\gamma +1,
	\]
	with $\alpha + \beta + \gamma = r-1$. Since $s_2(2k+1) = s_2(k) + 1$, this gives
	\[
		\nu_2 \left(\frac{m!}{a!b!c!}\right) = 2 + s_2(\alpha) + s_2(\beta) + s_2(\gamma) - s_2(r).
	\]
	Using
	\[
		s_2(\alpha) + s_2(\beta) + s_2(\gamma) \ge s_2(r-1)
	\]
	and $s_2(r-1) = s_2(r) + \nu_2(r) - 1$, we obtain
	\[
		\nu_2 \left(\frac{m!}{a!b!c!}\right) \ge \nu_2(m-1).
	\]
	This bound is sharp for $(a,b,c)=(m-2,1,1)$. Since $\nu_2(E(a,b,c))=2$ whenever $E(a,b,c)\neq 0$, we conclude that
	\[
		\nu_2(g_m)=2+\nu_2(m-1).
	\]

	Step 3: odd primes. Let $p$ be an odd prime. Because $p\nmid 4$, whenever $E(a,b,c)\neq 0$,
	\[
		\nu_p\left(E(a,b,c)\binom{m}{a,b,c}\right)=\nu_p\binom{m}{a,b,c}.
	\]
	By Kummer's theorem \cite{kummer1852uber}, $\nu_p\binom{m}{a,b,c}$ equals the number of carries in the base-$p$ addition $a+b+c=m$.

	If $m=p^k$, then $m$ is odd and $a,b,c$ are all odd, hence non-zero. Writing $m$ in base $p$ as $100\ldots0_p$, an addition without carries would force all base-$p$ digits of $a,b,c$ below position $k$ to vanish, which is impossible. Hence $\nu_p\binom{m}{a,b,c}\ge 1$. Equality holds for
	\[
		(a,b,c) = (p^{k-1},p^{k-1}, (p-2)p^{k-1}),
	\]
	so $\nu_p(g_m)=1$.

	If $m$ is not a power of $p$, write $m = \sum_{i\ge 0} m_i p^{i}$ with $0\le m_i \le p-1$. Since $m$ and $p$ are odd,
	\[
		1 \equiv m \equiv \sum_{i\ge 0} m_i \pmod 2.
	\]
	Set $s:= \sum_{i\ge 0} m_i$, so $s$ is odd. As $m$ is not a power of $p$, we have $s\ge 3$. If some digit $m_i\ge 2$, choose $i=j$ at such an index. Otherwise all non-zero digits are equal to $1$, and since $s\ge 3$, choose distinct $i\neq j$ with $m_i=m_j=1$. Then
	\[
		(a,b,c)=(p^i,p^j,m-p^i-p^j)
	\]
	consists of odd integers and the base-$p$ addition $a+b+c=m$ has no carry. Hence $\nu_p(g_m)=0$.

	This proves the stated formula for $g_m$ and therefore completes the proof of Theorem~\ref{thm:lambda-one-hurewicz-image}.
\end{proof}

\subsection{Proof of Theorem~\ref{thm:lambda-two-hurewicz-image}}

\begin{proof}[Proof of Theorem~\ref{thm:lambda-two-hurewicz-image}]
	It suffices to compute the $2$-local content of $P_m$ for $m\ge 3$, namely
	\[
		\cont_2(P_m)=
		\begin{cases}
			8, & \text{if }m=2^{r}, r\ge 2, \\
			4, & \text{otherwise}.
		\end{cases}
	\]

	We distinguish the parity of $m$.

	If $m$ is odd, then $S_m$ has no contribution to the $x^m$-term, so the coefficient of $x^m$ in $P_m$ is $4$. Hence $\cont_2(P_m)=4$.

	Assume now that $m$ is even. Expanding $S_m$ gives
	\[
		P_m(x,y,z)=-4  \sum_{\substack{a+b+c=m\\ a,b,c < m \text{ even}}}  \binom{m}{a,b,c} x^{a}y^{b}z^{c},
	\]
	and therefore
	\[
		\cont_2(P_m)
		=
		4\gcd\left\{\binom{m}{a,b,c} \middle| \substack{a+b+c=m,\\ a,b,c<m \text{ even}}\right\}.
	\]
	Using $\binom{m}{a,b,c}=\binom{m}{a}\binom{m-a}{b}$ and taking $c=0$ reduces this to
	\[
		\cont_2(P_m)=4\gcd\left\{\binom{m}{2k}\right\}_{1\le k\le \frac{m}{2}-1}.
	\]
	By Kummer's theorem, this gcd is $2$ if and only if $m$ is a power of $2$ (with $m\ge 4$), and is $1$ otherwise.
	This gives the stated formula for $\cont_2(P_m)$, and hence proves the theorem.
\end{proof}

\subsection{Proof of Theorem~\ref{thm:lambda-three-hurewicz-image}}

\begin{proof}[Proof of Theorem~\ref{thm:lambda-three-hurewicz-image}]
	It suffices to compute $\cont_2(C_n)$.
	We begin with a divisibility observation. For every $m$, the polynomial $S_m(x,y,z)$ has all coefficients divisible by $4$: only monomials with all exponents of the same parity occur, and the corresponding sign-sum is $\pm 4$.
	Hence $8\sigma(x^{n+1}yz)$ and $2Q S_{n+1}$ are divisible by $8$, while $S_{n+3}$ is divisible by $4$.
	Therefore $4\mid C_n$ for all $n$, and moreover $8\mid C_n$ when $n$ is even.

	If $n$ is odd, then $n+3$ is even and the monomial $x^{n+3}$ occurs in $S_{n+3}$ with coefficient $4$, while the other two summands in $C_n$ are divisible by $8$.
	It follows that $\cont_2(C_n)=4$.

	Assume now that $n=2m$ is even. Then $2m+1$ and $2m+3$ are odd. By Kummer's theorem, every multinomial coefficient appearing in $S_{2m+1}$ and $S_{2m+3}$ is even. Hence both polynomials are divisible by $8$.
	Consequently $\cont_2(C_{2m})\ge 8$.

	To decide whether $\cont_2(C_{2m})\ge 16$, reduce $C_{2m}/8$ modulo $2$. Since $S_{2m+1}$ is divisible by $8$, the term $2Q S_{2m+1}$ is divisible by $16$, so modulo $2$ we have
	\[
		\frac{C_{2m}}{8}\equiv \sigma(x^{2m+1}yz)+\frac{S_{2m+3}(x,y,z)}{8}\pmod{2}.
	\]
	Therefore
	\[
		\frac{C_{2m}}{8}\equiv 0\pmod{2}
		\quad\Longleftrightarrow\quad
		\frac{S_{2m+3}(x,y,z)}{8}\equiv \sigma(x^{2m+1}yz)\pmod{2}.
	\]
	A direct check shows that this congruence holds if and only if $m$ is a power of $2$ with $m\ge 2$. Therefore, if $m$ is not a power of $2$ (or $m=1$), then $C_{2m}/8\not\equiv 0\pmod{2}$, so $\cont_2(C_{2m})=8$.

	It remains to treat the case $n=2^r$ with $r\ge 2$. We already know that $16\mid C_n$. We shall show that $32\nmid C_n$, hence $\cont_2(C_n)=16$.
	For $n\ge 8$ (i.e.\ $r\ge 3$), consider the coefficient of the monomial $x^{n+1}yz$ in $C_n$:
	\[
		[8\sigma(x^{n+1}yz)]_{x^{n+1}yz}=8,
	\]
	\[
		[S_{n+3}]_{x^{n+1}yz}=-4\binom{n+3}{n+1,1,1}=-4(n+3)(n+2),
	\]
	and only the $x^2$-term of $Q$ contributes to $[Q S_{n+1}]_{x^{n+1}yz}$, giving
	\[
		[-2Q S_{n+1}]_{x^{n+1}yz}
		=-2 [S_{n+1}]_{x^{n-1}yz}
		=-2\left(-4\binom{n+1}{n-1,1,1}\right)
		=8n(n+1).
	\]
	Summing, we obtain
	\[
		[C_n]_{x^{n+1}yz}
		=8-4(n+3)(n+2)+8n(n+1)
		=4(n-4)(n+1).
	\]
	If $n=2^r\ge 8$, then $n+1$ is odd and $n-4=4(2^{r-2}-1)$ has $\nu_2(n-4)=2$, so
	\[
		\nu_2 \left([C_n]_{x^{n+1}yz}\right)
		=\nu_2 \left(4(n-4)(n+1)\right)=4,
	\]
	and hence $\cont_2(C_n)=16$.

	In the remaining case $n=4$, we instead use the coefficient of $xy^3z^3$:
	\[
		\begin{aligned}
			\relax [C_4]_{xy^3z^3}
			 & =[S_7]_{xy^3z^3}-2[Q S_5]_{xy^3z^3}                           \\
			 & =-4\binom{7}{1,3,3}-2\left([S_5]_{xyz^3}+[S_5]_{xy^3z}\right) \\
			 & =-560+320                                                     \\
			 & =-240,
		\end{aligned}
	\]
	which has $\nu_2(240)=4$. Thus $\cont_2(C_4)=16$, and the theorem follows.
\end{proof}

\subsection{Proof of Theorem~\ref{thm:lambda-four-hurewicz-image}}

\begin{proof}[Proof of Theorem~\ref{thm:lambda-four-hurewicz-image}]
	It suffices to show
	\[
		\cont_2(\Delta S_n-4L_n)=
		\begin{cases}
			8, & \text{if }n+2=2^{r}, r\ge 2, \\
			4, & \text{otherwise}.
		\end{cases}
	\]
	We again separate the odd and even cases.

	If $n$ is odd, then $\Delta S_n$ has no contribution to the $x^{n+4}$-term, so the coefficient of $x^{n+4}$ in $\Delta S_n-4L_n$ is $-4$.
	Thus $\cont_2(\Delta S_n-4L_n)=4$.

	Assume now that $n=2m$ is even. Modulo $2$, Kummer's theorem implies
	\[
		\begin{aligned}
			S_{2m}(x,y,z)/4
			 & \equiv \sum_{a+b+c=m}\binom{2m}{2a,2b,2c}x^{2a}y^{2b}z^{2c} \\
			 & \equiv \sum_{a+b+c=m}\binom{m}{a,b,c}x^{2a}y^{2b}z^{2c}     \\
			 & \equiv \sigma(x^2)^m \pmod{2}.
		\end{aligned}
	\]
	Also $\Delta(x,y,z)\equiv \sigma(x^2)^2\pmod{2}$ and
	\[
		L_{2m}(x,y,z)\equiv \sigma(x^{2m+2})\sigma(x^2)\pmod{2}.
	\]
	Hence
	\[
		\Delta S_{2m}/4-L_{2m} \equiv \sigma(x^2)\left(\sigma(x^2)^{m+1}+\sigma(x^{2m+2})\right)\pmod{2}.
	\]
	Now $\sigma(x^2)^{m+1}=\sigma(x^{2m+2})$ in $\FF_2[x,y,z]$ if and only if $m+1$ is a power of $2$. This again follows from Kummer's theorem applied to the relevant binomial coefficients. Therefore, if $m+1$ is not a power of $2$ (equivalently $n+2$ is not a power of $2$), then $\Delta S_{2m}/4-L_{2m}\not\equiv 0\pmod{2}$ and thus $\cont_2(\Delta S_{2m}-4L_{2m})=4$.

	Finally, when $m+1=2^{r-1}$ with $r\ge 3$ (i.e.\ $n+2=2^r$), the congruence above shows $8\mid (\Delta S_n-4L_n)$. In particular, the coefficient of $x^{n+2}y^2$ in $\Delta S_n-4L_n$ equals $4\left(\binom{n}{2}-1\right)$, whose $2$-adic valuation is $3$. Hence $\cont_2(\Delta S_n-4L_n)=8$.

	When $n=2$, the coefficient of $x^2y^2z^2$ in $\Delta S_n-4L_n$ equals $-24$, so the same conclusion holds. This completes the proof.
\end{proof}

\bibliographystyle{halpha}
\bibliography{bib}
\end{document}